\documentclass[bibother,reqno]{asl}

\usepackage[T1]{fontenc} 
\input glyphtounicode.tex
\usepackage{graphicx}
\usepackage{caption}
\usepackage{paralist}

\usepackage{xpatch}

\usepackage{amssymb,latexsym,stmaryrd} 
\usepackage{thmtools}

\usepackage[super]{nth}
\usepackage{nicefrac}

\usepackage[shortlabels]{enumitem}
\newlist{thmlist}{enumerate}{1} 
\setlist[thmlist]{label=(\roman{thmlisti})} 
\makeatletter 
\renewcommand{\p@thmlisti}{\perh@ps{\thethm}}
\protected\def\perh@ps#1{#1}

\newcommand{\itemref}[1]{%
  \begingroup 
  \let\perh@ps\@gobble\ref{#1}%
  \endgroup
}
\makeatother

\newcommand{\mybar}[1]{\bar{#1}} 
\newcommand{\secref}{\S\,}

\newcommand{\0}{\varnothing}
\newcommand{\acl}{\mathrm{acl}}
\newcommand{\Aut}{\mathrm{Aut}}
\newcommand{\C}{\mathcal{C}}
\newcommand{\dcl}{\mathrm{dcl}}
\newcommand{\dom}{\mathrm{dom}}
\newcommand{\E}{\mathcal{E}}
\newcommand{\eq}{^\mathrm{eq}}
\newcommand{\F}{\mathbb{F}}
\newcommand{\homsub}{\leq_\mathrm{hom}}
\newcommand{\Id}{\mathrm{Id}}

\newcommand{\La}{\mathcal{L}}
\newcommand{\M}{\mathcal{M}}
\newcommand{\N}{\mathcal{N}}
\newcommand{\Na}{\mathbb{N}}
\newcommand{\Npos}{\mathbb{N}^+}

\newcommand{\ptmem}{\pi} 
\newcommand{\pt}{\Phi} 
\newcommand{\Q}{\mathbb{Q}}

\newcommand{\Rpos}{\mathbb{R}^+}
\newcommand{\Rnn}{\mathbb{R}^{\geq 0}}
\newcommand{\rk}{\mathrm{rk}}
\newcommand{\substruc}{\leq}
\newcommand{\Th}{\mathrm{Th}}
\newcommand{\tp}{\mathrm{tp}}
\newcommand{\vphi}{\varphi}
\newcommand{\U}{\mathcal{U}}
\newcommand{\Z}{\mathbb{Z}}

\newbox\gnBoxA
\newdimen\gnCornerHgt
\setbox\gnBoxA=\hbox{$\ulcorner$}
\global\gnCornerHgt=\ht\gnBoxA
\newdimen\gnArgHgt
\def\Gnum #1{%
\setbox\gnBoxA=\hbox{$#1$}%
\gnArgHgt=\ht\gnBoxA%
\ifnum     \gnArgHgt<\gnCornerHgt \gnArgHgt=0pt%
\else \advance \gnArgHgt by -\gnCornerHgt%
\fi \raise\gnArgHgt\hbox{$\ulcorner$} \box\gnBoxA %
\raise\gnArgHgt\hbox{$\urcorner$}}

\theoremstyle{plain}
\newtheorem*{main}{Main result}
\newtheorem{thm}{Theorem}[subsection]
\newtheorem{lem}[thm]{Lemma}
\newtheorem{propn}[thm]{Proposition}
\newtheorem{cor}[thm]{Corollary}

\newtheorem{fact}[thm]{Fact}

\theoremstyle{definition}
\newtheorem{defn}[thm]{Definition}
\newtheorem{rem}[thm]{Remark}
\newtheorem{exmpl}[thm]{Example}
\newtheorem{nonexmpl}[thm]{Non-Example}
\newtheorem{question}[thm]{Question}

\renewcommand{\theequation}{\thesection.\arabic{equation}} 

\usepackage{color}
\usepackage{xcolor}
\colorlet{mylinkcolour}{blue!55!black}
\colorlet{mycitecolour}{blue!55!black}
\colorlet{myurlcolour}{blue!55!black}
\usepackage[hypertexnames=false]{hyperref} 
\usepackage[ocgcolorlinks]{ocgx2} 
\hypersetup{colorlinks=true,linkcolor={mylinkcolour},citecolor={mycitecolour},urlcolor={myurlcolour}}
\usepackage{nameref}
\usepackage{cite}
\usepackage[nameinlink]{cleveref}

\usepackage{chngcntr}

\usepackage{footnotebackref}

\makeatletter
\xpretocmd{\@maketitle}{\let\@makefntext\BHFN@OldMakefntext}{}{}
\renewcommand\@makefntext[1]{%
  \@ifundefined{@makefnmark}
    {}
    {%
     \renewcommand\@makefnmark{%
       \mbox{%
         \textsuperscript{%
           \normalfont
           \hyperref[\BackrefFootnoteTag]{\@thefnmark}%
         }%
       }\,%
     }%
     \BHFN@OldMakefntext{#1}%
  }%
}
\makeatother

\renewcommand{\subjclassname}{\textup{2020} Mathematics Subject Classification}

\renewcommand{\qedsymbol}{\mbox{$\square$}}

\makeatletter

\renewcommand*{\thmhead@plain}[2]{%
  \thmname{#1}\thmnumber{\@ifnotempty{#1}{ }\@upn{#2}}%
}

\renewcommand*{\swappedhead@plain}[2]{%
  \thmnumber{\@upn{#2}}\thmname{\@ifnotempty{#2}{. }#1}%
}

\def\@begintheorem#1#2[#3]{%
  \item[\normalfont 
  \hskip\labelsep
  \the\thm@headfont
  \thm@indent
  \@ifempty{#1}{\let\thmname\@gobble}{\let\thmname\@iden}%
  \@ifempty{#2}{\let\thmnumber\@gobble}{\let\thmnumber\@iden}%
  \thm@swap\swappedhead\thmhead{#1}{#2}]
  \@ifempty{#3}{{\hskip-\labelsep}\let\thmnote\@gobble}{\let\thmnote\@iden}%
  \thmnote{\textmd{\upshape(#3)}}%
  \the\thm@headpunct%
  \@restorelabelsep
  \thmheadnl 
  \ignorespaces}

\makeatother

\Crefname{chapter}{Chapter}{Chapters}
\Crefname{section}{\S\!}{\S\S\!}
\Crefname{subsection}{\S\!}{\S\S\!}
\Crefname{thm}{Theorem}{Theorems}
\Crefname{lem}{Lemma}{Lemmas}
\Crefname{propn}{Proposition}{Propositions}
\Crefname{cor}{Corollary}{Corollaries}
\Crefname{conj}{Conjecture}{Conjectures}
\Crefname{fact}{Fact}{Facts}
\Crefname{defn}{Definition}{Definitions}
\Crefname{rem}{Remark}{Remarks}
\Crefname{exmpl}{Example}{Examples}
\Crefname{nonexmpl}{Non-Example}{Non-Examples}
\Crefname{question}{Question}{Questions}
\crefname{footnote}{footnote}{footnotes}

\addtotheorempostheadhook[thm]{\crefalias{thmlisti}{thm}}
\addtotheorempostheadhook[lem]{\crefalias{thmlisti}{lem}}
\addtotheorempostheadhook[propn]{\crefalias{thmlisti}{propn}}
\addtotheorempostheadhook[cor]{\crefalias{thmlisti}{cor}}
\addtotheorempostheadhook[conj]{\crefalias{thmlisti}{conj}}
\addtotheorempostheadhook[fact]{\crefalias{thmlisti}{fact}}
\addtotheorempostheadhook[defn]{\crefalias{thmlisti}{defn}}
\addtotheorempostheadhook[rem]{\crefalias{thmlisti}{rem}}
\addtotheorempostheadhook[exmpl]{\crefalias{thmlisti}{exmpl}}
\addtotheorempostheadhook[nonexmpl]{\crefalias{thmlisti}{nonexmpl}}
\addtotheorempostheadhook[question]{\crefalias{thmlisti}{question}}

\title[Multidimensional exact classes]{Multidimensional exact classes, smooth approximation and bounded 4-types}
\author{Daniel Wolf}
\dedicatory{Dedicated to the memories of my son Arthur and my mother Valerie}
\thanks{The author (born Daniel Wood) was funded by the Leeds School of Mathematics through a Graduate Teaching Assistantship. The present work is taken largely from his PhD thesis \cite{wolfphd}.}
\address{\emph{F\MakeLowercase{ormerly of the}}\\School of Mathematics\\University of Leeds\\Leeds LS2 9JT\\United Kingdom}
\email{\href{mailto:dwolfeu@gmail.com}{dwolfeu@gmail.com}}

\begin{document}


\vspace*{-2\baselineskip}

\begin{center}
{\large To appear in the \href{http://aslonline.org/journals/the-journal-of-symbolic-logic/}{\emph{Journal of Symbolic Logic}}}

\vspace{.25\baselineskip}

{\footnotesize Submitted 2018-05-10 {\textperiodcentered} Accepted 2019-07-28 {\textperiodcentered} Resubmitted with revisions 2020-05-19}
\end{center}

\vspace{2\baselineskip}

\begin{abstract}
\hypertarget{abstract} In connection with the work of Anscombe, Macpherson, Steinhorn and the present author in \cite{amsw} we investigate the notion of a multidimensional exact class ($R$-mec), a special kind of multidimensional asymptotic class ($R$-mac) with measuring functions that yield the exact sizes of definable sets, not just approximations. We use results about smooth approximation \cite{klm} and Lie coordinatisation \cite{cherhrush} to prove the following result (\Cref{dugaldsconjecturefull}), as conjectured by Macpherson: For any countable language $\La$ and any positive integer $d$ the class $\C(\La,d)$ of all finite $\La$-structures with at most $d$ 4-types is a polynomial exact class in $\La$, where a polynomial exact class is a multidimensional exact class with polynomial measuring functions.
\end{abstract}

{\let\newpage\relax\maketitle}


\vspace*{-1.5\baselineskip}

\section{Introduction}
\label{section:intro}

The model-theoretic notion of an asymptotic class was introduced by Macpherson and Steinhorn in \cite{macstein1} as a generalisation of the result in \cite{cdm} of Chatzidakis, van den Dries and Macintyre regarding the size of definable sets in finite fields. This notion has been further generalised by Anscombe, Macpherson, Steinhorn and the present author in \cite{amsw} and \cite{wolfphd} to that of a multidimensional asymptotic class, also known as an $R$-mac. Details of the historical development of these notions can be found in \secref{1.1} of \cite{wolfphd}.

In the present work we focus on multidimensional exact classes, also known as $R$-mecs, which are a special kind of multidimensional asymptotic class where the measuring functions yield the exact sizes of definable sets, not just approximations. We show that multidimensional exact classes and smooth approximation (in the sense of \cite{klm}) are intimately related by proving that every smoothly approximable structure gives rise to a muldimensional exact class (\Cref{sapproxmecshort}). Using the framework of Lie coordinatisation, as developed by Cherlin and Hrushovski in \cite{cherhrush}, we then build on \Cref{sapproxmecshort} to prove the main result of this paper, as conjectured by Macpherson:

\begin{main}[\Cref{dugaldsconjecturefull}] For any countable language $\La$ and any positive integer $d$ the class $\C(\La,d)$ of all finite $\La$-structures with at most $d$ 4-types is a polynomial exact class in $\La$, where a polynomial exact class is a multidimensional exact class with polynomial measuring functions.
\end{main}

We outline the structure of the present work. In \Cref{section:mecs} we state the definition of an $R$-mec (and an $R$-mac), prove some technical lemmas and provide some examples and non-examples. \Cref{section:smoothapprox} is about smooth approximation and is where we prove the aforementioned \Cref{sapproxmecshort}. In \Cref{section:liecoord} we move on to Lie coordinatisation, which we use to prove the main result \Cref{dugaldsconjecturefull}.

We make extensive use of the Ryll-Nardzewski Theorem throughout this paper. This is well covered in the literature, for example \secref{1.3} of \cite{evans}, Theorem 7.3.1 in \cite{hodges}, Theorem 5.1 in \cite{kayemac} and Theorem 4.3.2 in \cite{tentziegler}. We refer to \cite{marker} and \cite{tentziegler} for general model-theoretic notation and terminology.


\section{Multidimensional exact classes}
\label{section:mecs}

We introduce the central definition of this paper, state and prove some handy lemmas in \Cref{subsec:usefullemmas} and then provide some (non-)examples in \Cref{subsec:examples}.

\subsection{Basic definitions}
\label{subsec:basicdefns}

Let $\La$ be a finitary, first-order language and let $\C$ be a class of finite $\La$-structures. For $m\in\Npos$ define
\[
\C(m):=\{(\M,\mybar{a}):\M\in\C,\mybar{a}\in M^m\}.
\]
We use $M$ (roman typeface) to denote the underlying set of the structure $\M$\linebreak (calligraphic typeface), although we do not maintain this distinction throughout. The elements of $\C(m)$ are sometimes referred to as \emph{pointed structures}. We define $\C(0):=\C$.

\begin{defn}[Definable partition]
Let $\pt$ be a partition of $\C(m)$. An element $\ptmem\in\pt$ is \emph{definable} if there exists a parameter-free $\La$-formula $\psi(\mybar{y})$ with $l(\mybar{y})=m$ such that for every $(\M,\mybar{a})\in\C(m)$ we have $(\M,\mybar{a})\in\ptmem$ if and only if $\M\models\psi(\mybar{a})$. The partition $\pt$ is \emph{definable} if $\ptmem$ is definable for every $\ptmem\in\pt$.
\end{defn}

The following definition is due to Anscombe, Macpherson, Steinhorn and the present author.

\begin{defn}[$R$-mec]\label{defnmec2}
Let $\C$ be a class of finite $\La$-structures and let $R$ be a set of functions from $\C$ to $\Na$. Then $\C$ is a \emph{multidimensional exact class for $R$ in $\La$}, or \emph{$R$-mec in $\La$} for short, if for every parameter-free $\La$-formula $\vphi(\mybar{x},\mybar{y})$, where $n:=l(\mybar{x}) \geq 1$ and $m:=l(\mybar{y})$, there exists a finite definable partition $\pt$ of $\C(m)$ such that for each $\ptmem\in\pt$ there exists $h_\ptmem\in R$ such that
\begin{equation}\label{sizeclause}
|\vphi(\M^n,\mybar{a})|=h_\ptmem(\M)
\end{equation}
for all $(\M,\mybar{a})\in\ptmem$.
\end{defn}

Before we provide the first example of an $R$-mec, we make some initial remarks and observations:

\begin{rem}~\label{macremarks}
\begin{thmlist}

\item\label{measuringfunctions} We call the functions $h_\ptmem$ the \emph{measuring functions} and the $\La$-formulas that define the partition $\pt$ the \emph{defining $\La$-formulas}. We often refer to multidimensional exact classes simply as \emph{exact classes}.

\item\label{differentvariables} In the $\La$-formula $\vphi(\mybar{x},\mybar{y})$ it is important to maintain a distinction between the free variables $\mybar{x}$ and the free variables $\mybar{y}$. (Although we use the plural \emph{variables}, either of $\mybar{x}$ and $\mybar{y}$ could denote a single variable.) The free variables $\mybar{x}$, which we call \emph{object variables}, are slots for solutions in each $\M\in\C$. The free variables $\mybar{y}$, which we call \emph{parameter variables}, are slots for parameters from each $\M\in\C$. To aid clarity we sometimes demarcate the two kinds of free variables with a semicolon, writing $\vphi(\mybar{x};\mybar{y})$. The \nameref{projlemmec} (\Cref{projlemmec}) shows that it suffices to consider formulas with only a single object variable.

\item We consider two edge cases where the definition holds trivially:

\begin{compactitem}

\item Suppose that $\vphi(\bar{x},\bar{y})$ is inconsistent, i.e.\ that $\vphi(\M^n,\bar{a}) = \0$ for every $(\M,\bar{a}) \in \C(m)$. Then the required definable partition of $\C(m)$ is the trivial partition $\{\C(m)\}$ and the measuring function is $\M \mapsto 0$.

\item Suppose that $m=0$, i.e.\ that the only free variables in $\vphi(\bar{x},\bar{y})$ are $\bar{x}$. Then $\C(m) = \C(0) = \C$ and $\vphi(\bar{x},\bar{y})$ can be written as $\vphi(\bar{x})$. The required definable partition of $\C$ is the trivial partition $\{\C \}$ and the measuring function is $\M \mapsto |\vphi(\M^n)|$.

\end{compactitem}
We will henceforth assume formulas to be consistent and $m$ to be positive.

\item $R$ must be closed under pointwise addition and multiplication: If $A$ and $B$ are definable sets, then their disjoint union $A\sqcup B$ is definable and has size $|A|+|B|$ and their cartesian product $A\times B$ is definable and has size $|A|\cdot |B|$.
So $R$ is generated under addition and multiplication by a subset of basic functions.

\item\label{defnweakmec} If we drop the requirement that the partition $\pt$ be definable, then we call $\C$ a \emph{weak $R$-mec}. We call \eqref{sizeclause} the \emph{size clause} and the requirement that the partition be definable the \emph{definability clause}. So a weak $R$-mec need satisfy only the size clause. We sometimes use the term {\it full $R$-mec} to emphasise that both the size and definability clauses hold and the term \emph{strictly weak $R$-mec} to emphasise that only the size clause holds.

\item\label{smallermac} $R$-mecs are closed under taking subclasses of $\C$ and supersets of $R$: If $\C$ is an $R$-mec in $\La$, then any subclass of $\C$ is also an $R'$-mec in $\La$ for any superset $R'\supseteq R$. Equivalently: If $\C$ is not an $R$-mec in $\La$, then no superclass of $\C$ is an $R'$-mec in $\La$ for any subset $R'\subseteq R$.

Weak $R$-mecs are closed under taking reducts: Let $\C$ be a weak $R$-mec in $\La$ and consider some $\La'\subseteq\La$. For $\M\in\C$, let $\M'$ denote the reduct of $\M$ to $\La'$. Then $\{\M':\M\in\C\}$ is a weak $R$-mec in $\La'$. Equivalently: Suppose that $\C$ is not a weak $R$-mec in $\La$ and consider some $\La'\supseteq\La$. For $\M\in\C$, let $\M'$ be an expansion of $\M$ to $\La'$. Then $\{\M':\M\in\C\}$ is not a weak $R$-mec in $\La'$. Note that we cannot remove the prefix `weak' here, since taking reducts may affect the definability clause.

\end{thmlist}
\end{rem}

We now provide our first class of examples. More examples are given in \Cref{subsec:examples}.

\begin{defn}\label{qedefn}
Let $\C$ be a class of finite $\La$-structures. We say that $\C$ has \emph{quantifier elimination in $\La$} if $\Th_\La(\M)$ has quantifier elimination for every $\M\in\C$, where $\Th_\La(\M)$ denotes the $\La$-theory of $\M$.
\end{defn}

\begin{propn}\label{qeexample}
Let $\La$ be a finite relational language and let $\C$ be a class of finite $\La$-structures. If $\C$ has quantifier elimination in $\La$, then there exists $R$ such that $\C$ is an $R$-mec in $\La$.
\end{propn}

\begin{proof}
Consider an $\La$-formula $\vphi(\bar{x},\bar{y})$, where $n:=l(\mybar{x}) \geq 1$ and $m:=l(\mybar{y})$. Let $A$ be the set of all $\La$-literals with free variables among $\bar{y}$. $A$ is finite because $\La$ is finite and relational. Let $B$ be the set of all maximally consistent conjunctions of literals from $A$. $B$ is finite because $A$ is finite. We can thus enumerate the elements of $B$ as $\psi_1(\bar{y}), \ldots, \psi_d(\bar{y})$ for some $d\in\Na^+$. Now consider some $\M\in\C$. Since $\C$ has quantifier elimination, $\Th_\La(\M)$ has quantifier elimination. Therefore each complete type in the free variables $\bar{y}$ is isolated by one of the $\psi_i(\bar{y})$. We define
\[
\pi_i := \{ (\M,\bar{a}) \in \C(m) : \M \models \psi_i(\bar{a})\}.
\]
Then $\{\pi_1, \ldots, \pi_d\}$ is a definable partition of $\C(m)$. Moreover, for each $i$, if $(\M,\bar{a}), (\M,\bar{b}) \in \pi_i$, then $\tp^\M(\bar{a}) = \tp^\M(\bar{b})$ and thus, since $\M$ is finite, there is an automophism $\sigma\colon\M\to\M$ such that $\sigma(\bar{a}) = \bar{b}$, which implies that
\[
|\vphi(\M^n,\bar{a})| = |\vphi(\M^n,\bar{b})|.
\]
So we may define $h_i(\M) := |\vphi(\M^n,\bar{a})|$, where $(\M,\bar{a}) \in \pi_i$. Then $h_i$ is the measuring function associated with $\pi_i$.
\end{proof}

\begin{cor}\label{simplestexample}
The class of finite sets is a multidimensional exact class in the language of pure equality.
\end{cor}

\begin{proof}
Let $\C$ denote the class of finite sets and $\La_=$ the language of pure equality. Since $\La_=$ is finite and relational and $\C$ has quantifier elimination in $\La_=$, we can apply \Cref{qeexample}.
\end{proof}

\begin{rem}
\Cref{simplestexample} underpins all other examples of multidimensional exact classes, since every language contains the language of pure equality as a sublanguage.
\end{rem}

\subsection{Asymptotic classes}
\label{subsec:ACs}

We provide the definitions of $N$-dimensional and multidimensional asymptotic classes, which we have already made reference to. The former is due to Macpherson and Steinhorn \cite{macstein1} and Elwes \cite{elwes}. The latter is due to Anscombe, Macpherson, Steinhorn and the present author \cite{amsw}.

\begin{defn}[$N$-dimensional asymptotic class]\label{NdimAC}
Let $\C$ be a class of finite $\La$-structures and let $N\in\Npos$. Then $\C$ is an {\it $N$-dimensional asymptotic class} if for every parameter-free $\La$-formula $\vphi(\mybar{x},\mybar{y})$, where $n:=l(\mybar{x}) \geq 1$ and $m:=l(\mybar{y})$, there exists a finite definable partition $\Phi$ of $\C(m)$ such that for each $\pi\in\Phi$ there exists $(d, \mu) \in (\{0,\ldots,Nn\}\times\Rpos)\cup\{(0,0)\}$ such that
\[
\left||\vphi(\M^n,\mybar{a})|\vphantom{\sum}-\mu|M|^{\nicefrac{d}{N}}\right|=o\left(|M|^{\nicefrac{d}{N}}\right)
\]
for all $(\M,\mybar{a})\in\Phi_{(d,\mu)}$ as $|M|\rightarrow\infty$, where the meaning of the little-o notation is as follows: For every $\varepsilon>0$ there exists $Q\in\Na$ such that for all $(\M,\mybar{a})\in\pi$, if $|M|>Q$, then
\[
\left||\vphi(\M^n,\mybar{a})|\vphantom{\sum}-\mu|M|^{\nicefrac{d}{N}}\right|\leq\varepsilon |M|^{\nicefrac{d}{N}}.
\]
We call $(d, \mu)$ a \emph{dimension--measure pair}.
\end{defn}

\begin{defn}[$R$-mac]\label{defnmac}
Let $\C$ be a class of finite $\La$-structures and let $R$ be a set of functions from $\C$ to $\Rnn$. Then $\C$ is a \emph{multidimensional asymptotic class for $R$ in $\La$}, or \emph{$R$-mac in $\La$} for short, if for every parameter-free $\La$-formula $\vphi(\mybar{x},\mybar{y})$, where $n:=l(\mybar{x}) \geq 1$ and $m:=l(\mybar{y})$, there exists a finite definable partition $\Phi$ of $\C(m)$ such that for each $\pi\in\Phi$ there exists $h_\pi\in R$ such \hypertarget{little-o-marker}that
\[
\left||\vphi(\M^n,\mybar{a})|\vphantom{\sum}-h_\pi(\M)\right|=o(h_\pi(\M))
\]
for all $(\M,\mybar{a})\in\pi$ as $|M|\rightarrow\infty$, where the meaning of the little-o notation is as follows: For every $\varepsilon>0$ there exists $Q\in\Na$ such that for all
$(\M,\mybar{a})\in\pi$, if $|M|>Q$, then
\[
\left||\vphi(\M^n,\mybar{a})|\vphantom{\sum}-h_\pi(\M)\right|\leq\varepsilon h_\pi(\M).
\]
\end{defn}

\begin{rem}~
\begin{thmlist}
\item The only difference bewteen an $N$-dimensional asymptotic class and an $R$-mac is the specification of the measuring functions, those of the former being restricted to the form $\mu|M|^{\nicefrac{d}{N}}$ while those of latter having no restriction on form.
\item In \Cref{defnmec2} the codomain of the functions in $R$ is $\Na$, while in \Cref{defnmac} it is $\Rnn$. The reason for this difference is the change from exact to approximate measuring functions.
\end{thmlist}
\end{rem}


\subsection{Useful lemmas}
\label{subsec:usefullemmas}

We state and prove a number of lemmas that we will use later on. We start with the Projection Lemma, which we already used in the proof of \Cref{simplestexample}.

\begin{lem}[Projection Lemma]\label{projlemmec}
Let $\C$ be a class of $\La$-structures. Suppose that the definition of an $R$-mec (\Cref{defnmec2}) holds for $\C$ and for all $\La$-formulas $\vphi(x,\mybar{y})$ with a single object variable $x$ (as opposed to a tuple $\mybar{x}$). Then $\C$ is an $R'$-mec in $\La$, where $R'$ is generated under addition and multiplication by the functions in $R$.
\end{lem}

A proof of the equivalent result for $R$-macs is given in \secref{2.4} of \cite{amsw}. It is adapted from the proof of Theorem 2.1 in \cite{macstein1}. Our proof of \Cref{projlemmec} is a simplified version of the proof in \cite{amsw}.

\begin{proof}[Proof of \Cref{projlemmec}]
Consider an arbitrary $\La$-formula $\vphi(\mybar{x},\mybar{y})$, where $n:=l(\mybar{x}) \geq 1$ and $m:=l(\mybar{y})$. We need to prove that it satisfies both the size and definability clauses. We do this by induction on the length of $\mybar{x}$. The base case of the induction is the hypothesis of the lemma.

Let $\mybar{x}=(x_1,\ldots, x_n)$. By the induction hypothesis we may assume that the size and definability clauses are satisfied by $\vphi(x_1,\ldots,x_{n-1};x_n,\mybar{y})$, where the semicolon is used to indicate the division between the object variables and the parameter variables (see \Cref{differentvariables}). So we have a finite partition $\Gamma$ of $\C(1+m)=\{(\M,a,\mybar{b}):\M\in\C, (a,\mybar{b})\in M^{1+m}\}$ with measuring functions $\{f_i:i\in\Gamma\}\subseteq R$ and defining $\La$-formulas $\{\gamma_i(x_n,\mybar{y}):i\in\Gamma\}$.

Consider each $\gamma_i(x_n,\mybar{y})$. By the base case of the induction, each $\gamma_i(x_n,\mybar{y})$ satisfies the size and definability clauses, so for each $i\in\Gamma$ we have a finite partition $\pt_i:=\{\ptmem_{i1},\ldots,\ptmem_{i{r_i}}\} $ of $\C(m)=\{(\M,\mybar{b}):\M\in\C,\mybar{b}\in M^m\}$ with measuring functions $\{g_{ij}:1\leq j\leq r_i\}\subseteq R$ and defining $\La$-formulas $\{\psi_{ij}(\mybar{y}):1\leq j\leq r_i\}$. We thus have $k:=|\Gamma|$ finite partitions of $\C(m)$. We use them to construct a single finite partition $\pt$ of $\C(m)$. Define
\[
\ptmem_{(j_1,\ldots,j_k)}:=\bigcap_{i\in\Gamma}\ptmem_{ij_{i}}\text{\ and\ }J:=\{(j_1,\ldots,j_k):1\leq j_i\leq r_i,1\leq i\leq k\}.
\]
Then $\pt:=\{\ptmem_{(j_1,\ldots,j_k)}:(j_1,\ldots,j_k)\in J\}$ forms a finite partition of $\C(m)$. We now need to show that this partition works.

We first consider the size clause. For each $\ptmem_{(j_1,\ldots,j_k)}$ we need to find a function $h_{(j_1,\ldots,j_k)}\in R$ such that
\begin{equation}\label{projlemeqn1}
h_{(j_1,\ldots,j_k)}(\M)=|\pt(\M^n,\mybar{b})|
\end{equation}
for all $(\M,\mybar{b})\in\ptmem_{(j_1,\ldots,j_k)}$. So fix some arbitrary $(j_1,\ldots,j_k)$ and consider an arbitrary pair $(\M,\mybar{b})\in\ptmem_{(j_1,\ldots,j_k)}$. (If $\ptmem_{(j_1,\ldots,j_k)}=\0$, then any function $h\in R$ would be vacuously suitable, so we can ignore this case.) Let $\chi_i(x_1,\ldots, x_n,\mybar{y})$ denote the $\La$-formula
\[
\vphi(x_1,\ldots, x_n,\mybar{y})\wedge\gamma_i(x_n,\mybar{y}).
\]
Then, since the $\La$-formulas $\gamma_i(x_n,\mybar{a})$ define the partition $\Gamma$, $\vphi(\M^n,\mybar{b})$ is partitioned by the $\chi_i(\M^n,\mybar{b})$, i.e.
\begin{equation}\label{projlemeqn2}
\vphi(\M^n,\mybar{b})=\bigcup_{i\in\Gamma}\chi_i(\M^n,\mybar{b}),
\end{equation}
where the union is disjoint. Now, for each $i\in\Gamma$ we have
\[
\left|\chi_i(\M^n,\mybar{b})\right| = \sum_{a\in\gamma_i(\M,\mybar{b})}\left|\vphi(\M^{n-1},a,\mybar{b})\right|
\]
because $\chi_i(\M^n,\mybar{b})$ fibres over $\gamma_i(\M,\mybar{b})$. Thus
\begin{equation}\label{projlemeqn3}
\left|\chi_i(\M^n,\mybar{b})\right| = f_i(\M)\cdot\left|\gamma_i(\M,\mybar{b})\right|,
\end{equation}
since $\left|\vphi(\M^{n-1},a,\mybar{b})\right|=f_i(\M)$ if $\M\models\gamma_i(a,\mybar{b})$. But $(\M,\mybar{b})\in\ptmem_{(j_1,\ldots,j_k)}\subseteq\ptmem_{i{j_i}}$ and so $\left|\gamma_i(\M,\mybar{b})\right|= g_{ij_i}(\M)$, which gives
\[
\left|\chi_i(\M^n,\mybar{b})\right| = f_i(\M)\cdot g_{ij_i}(\M)
\]
when put into \eqref{projlemeqn3}. Combining this with \eqref{projlemeqn2} yields
\[
\left|\vphi(\M^n,\mybar{b})\right| = \sum_{i\in\Gamma} f_i(\M)\cdot g_{ij_i}(\M).
\]
So define
\[
h_{(j_1,\ldots,j_k)}(\M):= \sum_{i=1}^k f_i(\M)\cdot g_{ij_i}(\M)
\]
for all $\M\in\C$ and \eqref{projlemeqn1} is satisfied as required.

We now come to the definability clause. Let $\psi_{(j_1,\ldots,j_k)}(\mybar{y})$ denote the formula
\[
\bigwedge_{i=1}^k \psi_{i{j_i}}(\mybar{y}).
\]
Then $(\M,\mybar{b})\in\ptmem_{(j_1,\ldots,j_k)}$ if and only if $\M\models\psi_{(j_1,\ldots,j_k)}(\mybar{b})$. So the definability clause is also satisfied and so we are done.
\end{proof}

The following lemma shows that $R$-mecs are closed under adding constant symbols:

\begin{lem}\label{adding_constants}
Suppose that $\C$ is an $R$-mec in $\La$. Let $\La'$ be an extension of $\La$ by constant symbols and for $\M\in\C$ let $\M'$ be an $\La'$-expansion of $\M$. Then $\C':=\{\M':\M\in\C\}$ is an $R$-mec in $\La'$.
\end{lem}

\begin{proof}
This follows straightforwardly from the definition of an $R$-mec.
\end{proof}

The following lemma shows that if we want to prove that a class $\C$ is an $R$-mec in $\La$, then it suffices to show that the definition eventually holds for each $\La$-formula:

\begin{lem}\label{finiteexceptions}
Suppose that the definition of a multidimensional exact class (\Cref{defnmec2}) holds for $\vphi(\mybar{x},\mybar{y})$, $R$ and the subclass
\[
\C(m)_{>Q}:=\{(\M,\mybar{a}):(\M,\mybar{a})\in\C(m)\text{ and }|M|>Q\}
\]
of $\C(m)$, where $m:=l(\mybar{y})$, $Q$ is some positive integer, and $R$ contains the constant function $\M\mapsto k$ for each positive integer $k\leq Q$. Then the definition also holds for $\vphi(\mybar{x},\mybar{y})$, $R$ and $\C(m)$.
\end{lem}

\begin{proof}
By the hypothesis of the lemma there exists a finite partition $\pt$ of $\C(m)_{>Q}$ with measuring functions $\{h_\ptmem:\ptmem\in\pt\}$ and defining $\La$-formulas $\{\psi_\ptmem(\mybar{y}):\ptmem\in\pt\}$. Let
\[
\Gamma_i:=\{(\M,\mybar{a}):\text{$\M\in\C(m)\setminus\C(m)_{>Q}$ and $|\vphi(\M^n,\mybar{a})|=i$}\}.
\]
Then $\{\Gamma_i:0\leq i\leq Q\}\cup\pt$ is a finite partition of $\C$ with measuring functions $\{g_i:0\leq i\leq Q\}\cup\{h_\ptmem:\ptmem\in\pt\}$, where $g_i(\M):=i$ for all $\M\in\C$. So the size clause holds for $\C$.

Let $\sigma_Q$ be the $\La$-sentence $\exists x_1\ldots\exists x_Q\forall y\bigvee_{1\leq i\leq Q}y=x_i$, i.e.\ $\sigma_Q$ says that there are at most $Q$ elements, and let $\vphi_i(\mybar{y})$ be the $\La$-formula $\exists!_i\mybar{x}\,\vphi(\mybar{x},\mybar{y})$, i.e.\ $\vphi_i(\mybar{a})$ says that $|\vphi(\M^n,\mybar{a})|=i$. Then the partition in the previous paragraph is defined by the $\La$-formulas $\{\vphi_i(\mybar{y})\wedge\sigma_Q:1\leq i\leq Q\}\cup\{\psi_\ptmem(\mybar{y})\wedge\neg\sigma_Q:\ptmem\in\pt\}$.
\end{proof}

Our last useful lemma is a compactness-like result:

\begin{lem}\label{reductmec}
Let $\C$ be a class of finite $\La$-structures. For $\La'\subseteq\La$ let $\C_{\La'}$ denote the class of all $\La'$-reducts of structures in $\C$. If $\C_{\La'}$ is an $R$-mec in $\La'$ for every finite $\La'\subseteq\La$, then $\C$ is an $R$-mec in $\La$.
\end{lem}

\begin{proof}
This follows from \Cref{defnmec2}, whose first (second-order) quantifier ranges over $\La$-formulas, and the following two facts: Firstly, $\La$-formulas are finite and so any $\La$-formula is an $\La'$-formula for some finite $\La'\subseteq\La$. Secondly, for every $\La'$-formula $\chi(\mybar{y})$ (where $m:=l(\mybar{y}$), for every $\La'$-reduct $\M'$ of an $\La$-structure $\M$ and for every $\mybar{a}\in M^m$, $\M'\models\chi(\mybar{a})$ if and only if $\M\models\chi(\mybar{a})$.
\end{proof}


\subsection{Examples and non-examples}
\label{subsec:examples}

Following on from \Cref{simplestexample}, we provide a number of examples and non-examples of R-mecs. In order to explain the first example (\Cref{cyclicdirectsum}) we require a definition and a lemma:

\begin{defn}[Disjoint union of classes]\label{disjointunionclasses}
Consider $\C_1,\ldots,\C_k$, where each $\C_i$ is a class of $\La_i$-structures. Define the \emph{disjoint union} of $\C_1,\ldots,\C_k$ to be
\[
\C_1\sqcup\dots\sqcup\C_k:=\{\M_1\sqcup\dots\sqcup\M_k:\M_i\in\C_i\},
\]
where we define a first-order structure on $\M_1\sqcup\dots\sqcup\M_k$ as follows: The domain is $M_1\cup\dots\cup M_k$, which we make formally disjoint if necessary. The language is $\La_1\sqcup\dots\sqcup\La_k$, which has a sort $S_i$ for each $M_i$ and contains all $\La_i$-symbols for every $i\in\{1,\ldots,k\}$, with each $\La_i$-symbol being restricted to the sort $S_i$.
\end{defn}

\begin{lem}\label{disjointunionclasseslem}
Let $\C_i$ be an $R_i$-mec in $\La_i$. Then $\C_1\sqcup\dots\sqcup\C_k$ is an $R$-mec in $\La:=\La_1\sqcup\dots\sqcup\La_k$, where $R$ is the set generated by $R_1\cup\cdots\cup R_k$ under addition and multiplication.
\end{lem}

\begin{proof}
We restrict our attention to the case $k=2$, the general case following by induction.

Consider an $\La$-formula $\vphi(\mybar{x}_1,\mybar{x}_2;\mybar{y}_1,\mybar{y}_2)$, where $\mybar{x}_i$ and $\mybar{y}_i$ are of sort $S_i$. By an induction on the complexity of the formula, one can show that $\vphi(\mybar{x_1},\mybar{x_2};\mybar{y_1},\mybar{y_2})$ is equivalent to a finite disjunction of $\La$-formulas of the form $\chi(\mybar{x}_1,\mybar{y}_1)\wedge\theta(\mybar{x}_2,\mybar{y}_2)$, where $\chi$ is an $\La_1$-formula, $\theta$ is an $\La_2$-formula, and the disjuncts are pairwise inconsistent. Since the domains of $\M_1\in\C_1$ and $\M_2\in\C_2$ are disjoint, we have
\[
|\chi(\M_1\sqcup\M_2,\mybar{a}_1)\wedge\theta(\M_1\sqcup\M_2,\mybar{a}_2)|=|\chi(\M_1,\mybar{a}_1)|\cdot |\theta(\M_2,\mybar{a}_2)|.
\]
One then proceeds by using the facts that the disjuncts are pairwise inconsistent, thus allowing summation, and that each $\C_i$ is an $R$-mec.
\end{proof}

\begin{exmpl}\label{cyclicdirectsum}
Consider the class $\C$ of finite cyclic groups and for arbitrary $k\in\Npos$ define $\C_k:=\{C_1\oplus\dots\oplus C_k:C_i\in\C\}$. Let $\La$ be the language of groups (with or without a constant symbol for the identity element -- recall \Cref{adding_constants}). Then $\C_k$ is a multidimensional exact class in $\La'$, where $\La'$ is $\La$ adjoined with a unary predicate $P_i$ for each part of the direct sum:
\[
{P_i}^{C_1\oplus\dots\oplus C_k}:=\{(0,\ldots,0,\hspace{-1em}\underset{\underset{\text{$i^\mathrm{th}$ place}}{\uparrow}}{a}\hspace{-0.9em},0,\ldots,0):a\in C_i\}.
\]
\end{exmpl}

\begin{proof}
Theorem 3.14 in \cite{macstein1} states that $\C$ is a $1$-dimensional asymptotic class in $\La$ (see \Cref{NdimAC}). Inspection of the proof of this theorem shows that $\C$ is in fact an exact class, since the measuring functions yield exact sizes and not just approximaitons. So by \Cref{disjointunionclasseslem}, $\underbrace{\C\sqcup\cdots\sqcup\C}_{\text{$k$ times}}$ is an exact class in $\underbrace{\La\sqcup\cdots\sqcup\La}_{\text{$k$ times}}$. We now use the work in \cite{amsw} and \secref{2.4} of \cite{wolfphd} regarding interpretability: Since $\La'$ is equipped with the predicates $P_i$, it follows that $\C_k$ and $\C\sqcup\cdots\sqcup\C$ are $\0$-bi-interpretable and thus that $\C_k$ is an exact class.
\end{proof}

\begin{rem} We comment on \Cref{cyclicdirectsum}.
The class $\C$ of finite cyclic groups is both a multidimensional exact class and a $1$-dimensional asymptotic class, so one might wonder whether it could be a ``$1$-dimensional exact class''. However, the notion of an $N$-dimensional exact class is inconsistent: Consider two disjoint definable sets $A,B\subseteq M$ with $|A|=\alpha|M|^{a/N}$ and $|B|=\beta|M|^{b/N}$, where $a>b$. Then their union $A\cup B$, which is definable, has size $\alpha|M|^{a/N}+\beta|M|^{b/N}$, which cannot be expressed in the form $\mu|M|^{d/N}$ for a dimension--measure pair $(d,\mu)$. This is not an issue for an $N$-dimensional \emph{asymptotic} class, since $|M|^{a/N}$ swamps $|M|^{b/N}$ as $|M|\rightarrow\infty$. It is also not an issue for a multidimensional exact class, where one is not bound to dimension--measure pairs.
\end{rem}

\begin{exmpl}[Proposition 4.4.2 in \cite{gms}]\label{gmsexmpl}
Consider the class of homocyclic groups
\[
\C:=\{(\Z/p^n\Z)^m:\text{$p$ is prime and\ }n,m\in\Npos\}
\]
in the language $\La:=\{+\}$. This class is an $R$-mec, where $R$ is generated by functions of the form
\[
\sum_{i=0}^r\sum_{j=-rd}^{rd}c_{ij}p^{m(in+j)},
\]
where $r$ is the length of the object-variable tuple of the given $\La$-formula (see \Cref{differentvariables}); $d$ is a positive integer that is constructively determined by the $\La$-formula; and the $c_{ij}$ are integers that depend on the $\La$-formula, with $c_{ij}:=0$ whenever $in+j<0$. (Each group $(\Z/p^n\Z)^m\in\C$ is determined by a triple $(p,n,m)$, so by defining a function on such triples we also define a function on $\C$.)
\end{exmpl}

The following two examples are taken from \cite{amsw}:

\begin{exmpl}\label{r-modules}
Let $R$ be a ring and let $\C$ be the class of all finite $R$-modules. Then there exists $R'$ such that $\C$ is an $R'$-mec.
\end{exmpl}

\begin{exmpl}\label{abelian-groups}
There exists $R$ such that the class of finite abelian groups is an $R$-mec.
\end{exmpl}

Further examples will arise as this paper progresses. We now turn our attention to non-examples, which are often just as interesting.

\begin{nonexmpl}[Example 3.1 in \cite{macstein1}]\label{orderingnonexample}
The class $\C$ of all finite linear orders in (any extension of) the language $\La=\{<\}$ does not form a weak $R$-mec for any $R$.
\end{nonexmpl}

\begin{proof}
Let $\vphi(x,y)$ be the formula $x<y$ and consider the finite linear order $\M_k:=\{a_0<\cdots <a_k\}$. Then $|\vphi(\M_k,a_i)|=i$. As we let $k$ increase and let $i$ vary we define arbitrarily many subsets of distinct sizes. Thus no finite number of functions can yield $|\vphi(\M_k,a_i)|$ for all $k,i\in\Na$. Let's make that argument a little more rigorous.

By way of contradiction, suppose that there exists $R$ such that $\C$ forms a weak $R$-mec. So for the formula $\vphi(x,y)$ there exists a finite partition $\pt$ of $\C(1)$ with measuring functions $\{h_\ptmem:\ptmem\in\pt\}\subseteq R$. Let $t:=|\pt|$ and consider the finite linear order $\M_t$. Then $t$ measuring functions are not enough for this structure, since there are $t+1$ different sizes of the definable subsets, namely $|\vphi(\M_t,a_0)|=0,\ldots,|\vphi(\M_t,a_t)|=t$. A contradiction.
\end{proof}

The following non-example is informative, as it shows that the choice of language in \Cref{gmsexmpl} is important:

\begin{nonexmpl}\label{gmsnonexmpl}
Let $p$ be prime. Then the class $\{\Z/p^n\Z:n\in\Npos\}$ of multiplicative monoids in (any extension of) the language $\La=\{\times\}$ does not form a weak $R$-mec for any $R$.
\end{nonexmpl}

\begin{proof}
Let $R$ be any set of functions from $\C$ to $\Na$ and let $\vphi(x,y)$ be the formula $\exists z\,(x=z\times y)$. Then $|\vphi(\Z/p^n\Z,p^i)|=p^{n-i}$. So as we let $n$ increase and let $i$ vary we define arbitrarily many subsets of distinct sizes. Thus, by the same argument given in the proof of \Cref{orderingnonexample}, no finite number of measuring functions can suffice for $|\vphi(\Z/p^n\Z,p^i)|$ for all $n,i\in\Na$.
\end{proof}

\begin{rem}~
\begin{thmlist}
\item \Cref{orderingnonexample,gmsnonexmpl} are special cases of the general fact that an ultraproduct of a weak multidimensional asymptotic class cannot have the strict order property; see \cite{amsw}. (See Definition 2.14 in \cite{casanovas} or Exercise 8.2.4 in \cite{tentziegler} for a definition of the strict order property.) Note that we do mean the \emph{strict} order property here. For example, the Paley graphs form an asymptotic class (Example 3.4 in \cite{macstein1}), but any ultraproduct of them has unstable theory (see \Cref{paleyrem}).
\item The issue preventing \Cref{gmsnonexmpl} from being an $R$-mec is the unbounded exponent $n$. If the exponent is bounded, then one can have an $R$-mec, as shown by the work of Bello Aguirre in \cite{bello} and \cite{bellophd}.
\end{thmlist}
\end{rem}

We now cite two non-examples concerning ultraproducts, the random graph and the random tournament,\footnote{Due to its different guises, the random graph goes by various names, including the `Rado graph' and `the generic (countable homogeneous) graph'. The random tournament has similar aliases.} which are covered extensively in the literature, for instance \cite{bellslom}, Exercise 2.5.19 in \cite{marker} and Exercise 1.2.4 in \cite{tentziegler} (ultraproducts), p.~232 of \cite{cher88}, p.~17 of \cite{cher98}, \S\secref{1--2} of \cite{evans}, p.~435 of \cite{lach84b}, pp.~50--52 of \cite{marker} and Exercise 3.3.1 in \cite{tentziegler} (the random graph and the random tournament). We cite these two non-examples in order to highlight a difference between multidimensional exact classes and multidimensional asymptotic classes (\Cref{paleyrem}).

\begin{nonexmpl}[\hspace{-.05em}\cite{amsw} or Non-Example 2.3.12 in \cite{wolfphd}]\label{randomgraphnonexmpl}
The random graph is not elementarily equivalent to an ultraproduct of a multidimensional exact class.
\end{nonexmpl}

\begin{nonexmpl}[\hspace{-.05em}\cite{amsw} or Non-Example 2.3.14 in \cite{wolfphd}]\label{randomtournamentnonexmpl}
The random tournament is not elementarily equivalent to an ultraproduct of a multidimensional exact class.
\end{nonexmpl}

(Due to a typesetting error, the tournament relation $a\rightarrowtriangle b$ is incorrectly displayed as $a\,\raisebox{4pt}{$\cdot$}\, b$ in Non-Example 2.3.14 in \cite{wolfphd}.)

\begin{rem}\label{paleyrem}
The situation is quite different for asymptotic classes: The random graph is elementarily equivalent to any infinite ultraproduct of the class of Paley graphs, which is a $1$-dimensional asymptotic class (Example 3.4 in \cite{macstein1}), and the random tournament is elementarily equivalent to any infinite ultraproduct of the class of Paley tournaments, which is also a $1$-dimensional asymptotic class (Example 3.5 in \cite{macstein1}). This is an interesting phenomenon, especially in light of Theorem 7.5.6 in \cite{cherhrush} and \Cref{dugaldsconjecturefull}. We will discuss it further in \Cref{bipartitequestion}.
\end{rem}


\section{Smooth approximation and exact classes}
\label{section:smoothapprox}

The goal of this section is to prove \Cref{sapproxmecshort}, which states that finite structures smoothly approximating an $\aleph_0$-categorical structure form a multidimensional exact class. In \Cref{subsec:smoothapprox} we define the notion of smooth approximation and then provide some examples. In \Cref{subsec:smoothapproxexact} we state and prove the result.

\subsection{Smooth approximation}
\label{subsec:smoothapprox}

The notion of smooth approximation was introduced by Lachlan in the 1980s, arising as a generalisation of $\aleph_0$-categorical, $\aleph_0$-stable structures \cite{chl}, in particular Corollary 7.4 of that paper. \cite{cl}, \cite{kl}, \cite{lach84}, \cite{lach87} and \cite{lachsurvey} are also relevant, but the key texts on smooth approximation itself are \cite{klm} by Kantor, Liebeck and Macpherson and \cite{cherhrush} by Cherlin and Hrushovski. A history of the development of the notion is to be found in \secref{1.1} of \cite{cherhrush} and there is a survey article \cite{macsmooth}, which also contains improvements and errata to \cite{klm}. Smooth approximation also arises in the context of asymptotic classes in \cite{elwesphd}, \cite{elwes}, \cite{macstein1} and \cite{macstein2}.

For $\La$-structures $\M$ and $\N$ we use the notation $\N\substruc\M$ to mean that $\N$ is an $\La$-substructure of $\M$.

\begin{defn}[Homogeneous substructure]\label{homsubstruc}
Let $\M$ and $\N$ be $\La$-structures. $\N$ is a \emph{homogeneous substructure}\footnote{We define `homogeneous substructure' as one term, not as the conjunction of two words; that is, `homogeneous substructure' does not mean a substructure that is homogeneous.} of $\M$, notationally $\N\homsub\M$, if $\N\substruc\M$ and for every $k\in\Npos$ and every pair $\mybar{a},\mybar{b}\in N^k$, $\mybar{a}$ and $\mybar{b}$ lie in the same $\Aut(\M)$-orbit if and only if $\mybar{a}$ and $\mybar{b}$ lie in the same $\Aut_{\{ N\}}(\M)$-orbit, where
\[
\Aut_{\{ N\}}(\M):=\{\sigma\in\Aut(\M):\sigma (N)=N\}.
\]
\end{defn}

\begin{defn}[Smooth approximation]\label{defnsmoothapprox}
An $\La$-structure $\M$ is \emph{smoothly approximable} if $\M$ is $\aleph_0$-categorical and there exists a sequence $(\M_i)_{i<\omega}$ of finite homogeneous substructures of $\M$ such that $M_i\subset M_{i+1}$ for all $i<\omega$ and $\bigcup_{i<\omega}M_i=M$. We say that $\M$ is \emph{smoothly approximated} by the $\M_i$.
\end{defn}

We provide some examples of smoothly approximable structures, starting with a trivial example:

\begin{exmpl}
Let $\M$ be a countably infinite set in the language of equality. Enumerate $\M$ as $(a_i:i<\omega)$ and let $\M_i=\{a_0,\ldots,a_i\}$. Then each $\M_i$ is a finite homogeneous substructure of $\M$ and $\M=\bigcup_{i<\omega}\M_i$.
\end{exmpl}

\begin{exmpl}\label{smoothapproxexmpl1}
Consider a language $\La:=\{I_1,I_2\}$, where $I_1$ and $I_2$ are binary relation symbols. Let $\M$ be a countable $\La$-structure where $I_1^\M$ and $I_2^\M$ are equivalence relations such that $I_1^\M$ has infinitely many classes, $I_2^\M$ refines $I_1^\M$, every $I_1$-equivalence class contains infinitely many $I_2$-equivalence classes, and every $I_2$-equivalence class is infinite; that is, $\M$ is partitioned into infinitely many $I_1$-equivalence classes, each of which is then partitioned into infinitely many $I_2$-equivalence classes, each of which is infinite. Note that $\M$ is unique up to isomorphism and hence $\aleph_0$-categorical, since the structure is first-order expressible in $\La$.

Enumerate the $I_1$-equivalence classes as $(a_i:i<\omega)$ and the $I_2$-equivalence classes within each $a_i$ as $(a_{ij}:j<\omega)$. Finally, enumerate the elements of each $a_{ij}$ as $(a_{ijk}:k<\omega)$. Let $\M_{(r,s,t)}:=\{a_{ijk}:i\leq r, j\leq s,k\leq t\}$. Then each $\M_{(r,s,t)}$ is a finite homogeneous substructure of $\M$ and $\M=\bigcup_{r<\omega}\M_{(r,r,r)}$.

Note that this example straightforwardly generalises to the case of $n$ nested equivalence relations for any $n<\omega$.
\end{exmpl}

\begin{exmpl}\label{smoothapproxexmpl2}
Let $\M$ be the direct sum of $\omega$-many copies of the additive group $\Z/p^2\Z$, where $p$ is some fixed prime. Note that $\M$ is $\aleph_0$-categorical, which can be seen via Szmielew invariants (see Appendix A.2 in \cite{hodges}). Let $\M_i$ consist of the first $i$ copies of $\Z/p^2\Z$. Then each $\M_i$ is a finite homogeneous substructure of $\M$ and $\M=\bigcup_{i<\omega}\M_i$.
\end{exmpl}


\subsection{Smooth approximation is exact}
\label{subsec:smoothapproxexact}

We now come to \Cref{sapproxmecshort}, the central result of this section. We first give the main proof, leaving the necessary technical lemmas until afterwards.

\begin{propn}\label{sapproxmecshort}
Let $\M$ be an $\La$-structure smoothly approximated by finite homogeneous substructures $(\M_i)_{i<\omega}$. Then there exists $R$ such that $\C:=\{\M_i:i<\omega\}$ is an $R$-mec in $\La$.
\end{propn}

\begin{proof}
Let $\vphi(\mybar{x},\mybar{y})$ be an $\La$-formula with $n:=l(\mybar{x}) \geq 1$ and $m:=l(\mybar{y})$.

We first cover the size clause. We use the Ryll-Nardzewski Theorem: Since $\M$ is $\aleph_0$-categorical, $\Aut(\M)$ acts oligomorphically on $\M$ and thus $\M^m$ has only finitely many $\Aut(\M)$-orbits, say $\Theta_1,\ldots,\Theta_d$. We use these orbits to define a finite partition $\ptmem_1,\ldots,\ptmem_d$ of $\C(m)=\{(\M_i,\mybar{a}):i<\omega,\mybar{a}\in {M_i}^m\}$:
\[
(\M_i,\mybar{a})\in\ptmem_j \,\text{ iff }\, \mybar{a}\in\Theta_j.
\]
Define $\ptmem_j^{\M_i}:=\{\mybar{a}\in {M_i}^m:(\M_i,\mybar{a})\in\ptmem_j\}$ and let $\mybar{a},\mybar{b}\in {M_i}^m$. Then
\begin{equation}\label{crucialimplication}
\begin{aligned}
\mybar{a},\mybar{b}\in\ptmem_j^{\M_i} & \iff \mybar{a},\mybar{b}\in\Theta_j\\
& \,\implies \text{$\mybar{a}$ and $\mybar{b}$ lie in the same $\Aut_{\{ M_i\}}(\M)$-orbit}\\
& \hphantom{\iff}\,\text{(since $\M_i\homsub\M$)}\\
 & \,\implies |\vphi({\M_i}^n,\mybar{a})|=|\vphi({\M_i}^n,\mybar{b})|.
\end{aligned}
\end{equation}
We justify the last implication: Since $\mybar{a}$ and $\mybar{b}$ lie in the same $\Aut_{\{ M_i\}}(\M)$-orbit, there is some $\sigma\in\Aut_{\{ M_i\}}(\M)$ such that $\sigma(\mybar{a})=\mybar{b}$. But $\sigma\restriction M_i$ is an automorphism of $\M_i$ and thus $\M_i\models\vphi(\mybar{c},\mybar{a})$ if and only if $\M_i\models\vphi(\sigma(\mybar{c}),\sigma(\mybar{a}))$. Therefore $\sigma\colon\vphi({\M_i}^n,\mybar{a})\to\vphi({\M_i}^n,\mybar{b})$ is a bijection and hence $|\vphi({\M_i}^n,\mybar{a})|=|\vphi({\M_i}^n,\mybar{b})|$.

Define $h_j(\M_i):=|\vphi({\M_i}^n,\mybar{a})|$, where $\mybar{a}$ is some arbitrary element of $\ptmem_j^{\M_i}$ (if no such $\mybar{a}$ exists, then the value of $h_j$ at $\M_i$ can be chosen to be anything, say 0); this function is well-defined by \eqref{crucialimplication}. Then $\ptmem_1,\ldots,\ptmem_d$ and $h_1,\ldots,h_d$ satisfy the size clause.

We now come to the definability clause. We use the Ryll-Nardzewski Theorem again: Each orbit $\Theta_j$ is the solution set of an isolated $m$-type and so the $\La$-formula isolating this type defines $\Theta_j$ in $\M$; let $\psi_j(\mybar{y})$ be the isolating formula for $\Theta_j$. So $\M\models\psi_j(\mybar{a})$ if and only if $\mybar{a}\in\Theta_j$. We claim that the following is eventually true, i.e.\ there exists $Q\in\Na$ such that for each $\psi_j$, if $i>Q$, then
\begin{equation}\label{eventualequiv}
\M_i\models\psi_j(\mybar{a}) \iff \mybar{a}\in\ptmem_j^{\M_i}
\end{equation}
for every $\mybar{a}\in {M_i}^m$. By \Cref{finiteexceptions} this suffices to prove the definability clause.

We prove this claim: Apply \Cref{eventuallyiff} to $\psi_j$ to obtain $Q_j\in\Na$ such that if $i>Q_j$ and $\mybar{a}\in {M_i}^m$, then
\begin{equation}\label{Qj}
\M\models\psi_j(\mybar{a}) \iff \M_i\models\psi_j(\mybar{a}).
\end{equation}
Let $Q:=\max\{Q_j:1\leq j\leq k\}$. Consider $\mybar{a}\in {M_i}^m$ with $i>Q$. Then
\[
\M_i\models\psi_j(\mybar{a}) \overset{\eqref{Qj}}{\vphantom{\in}\iff\vphantom{\in}} \M\models\psi_j(\mybar{a}) \iff \mybar{a}\in\Theta_j \iff \mybar{a}\in\ptmem_j^{\M_i}
\]
and so \eqref{eventualequiv} holds.
\end{proof}

\begin{rem}
The proofs of \Cref{qeexample,sapproxmecshort} rest on the same property, namely the existence of a uniform bound on the number of types in each structure in $\C$. In the proof of \Cref{qeexample} this uniformity arises from the language $\La$ directly: We found the isolating formulas $\psi_1(\mybar{y}),\ldots,\psi_d(\mybar{y})$ before considering structures in $\C$. In the proof of \Cref{sapproxmecshort} this uniform{\-}ity arises from the oligomorphicity of $\M$, which is then passed down to the homogeneous substructures $\M_i$.
\end{rem}

\begin{defn}[Canonical language]\label{defncanonlang} We define the \emph{canonical language}\,\footnote{Note that the term \emph{canonical language} is sometimes used to refer to the smaller language $\La^*\setminus\La$. We avoid this usage.} of $\M$ to be
\[
\La^*:=\La\cup\{P_\Theta:\text{$\Theta$ is a $\Aut(\M)$-orbit of $\M$}\},
\]
where each $P_\Theta$ is a new unary predicate symbol. We expand $\M$ to an $\La^*$-structure $\M^*$ by defining the assignment of each $P_\Theta$ in $\M^*$ to be $\Theta$. We expand each $\M_i$ to an $\La^*$-structure ${\M_i}^*$ by defining the assignment of each $P_\Theta$ to be $\Theta\cap M_i$.
\end{defn}

\Cref{sameauto,smoothapproxqe} are standard and we state them without proof:

\begin{lem}\label{sameauto}
$\Aut(\M)=\Aut(\M^*)$.
\end{lem}

\begin{lem}\label{smoothapproxqe}
$\Th(\M^*)$ has quantifier elimination; in particular, any $\La^*$-formula is equivalent in $\Th(\M^*)$ to a quantifier-free $(\La^*\setminus\La)$-formula.
\end{lem}

\begin{lem}\label{stillsmoothapprox}
 $\M^*$ is smoothly approximated by $({\M_i}^*)_{i<\omega}$.
\end{lem}

\begin{proof}
Since $\M$ is $\aleph_0$-categorical, by \Cref{sameauto} and the Ryll-Nardzewski Theorem, $\M^*$ is also $\aleph_0$-categorical. Also note that each ${\M_i}^*$ is a finite $\La^*$-substructure of $\M^*$. It remains to show that ${\M_i}^*\homsub\M^*$. If $\mybar{a},\mybar{b}\in{\M_i}^*$ lie in the same $\Aut(\M^*)_{\{M_i\}}$-orbit, then $\mybar{a}$ and $\mybar{b}$ lie in the same $\Aut(\M^*)$-orbit, since $\Aut(\M^*)_{\{M_i\}}\subseteq\Aut(\M^*)$. Now suppose that $\mybar{a},\mybar{b}\in{\M_i}^*$ lie in the same $\Aut(\M^*)$-orbit. By \Cref{sameauto}, $\mybar{a}$ and $\mybar{b}$ lie in the same $\Aut(\M)$-orbit. Thus, since $\M_i\homsub\M$, there exists $\sigma\in\Aut(\M)_{\{M_i\}}$ such that $\sigma(\mybar{a})=\mybar{b}$. But $\sigma\in\Aut(\M^*)_{\{M_i\}}$, again by \Cref{sameauto}, and so $\mybar{a}$ and $\mybar{b}$ lie in the same $\Aut(\M^*)_{\{M_i\}}$-orbit.
\end{proof}

\begin{lem}\label{eventuallyiff}
Let $\chi(\mybar{y})$ be an $\La$-formula with $m:=l(\mybar{y})$. Then there exists $Q\in\Na$ such that if $i>Q$ and $\mybar{c}\in {M_i}^m$, then
\[
\M\models\chi(\mybar{c})\iff\M_i\models\chi(\mybar{c}).
\]
\end{lem}

\begin{proof}
Consider $\M^*$. By \Cref{smoothapproxqe}, $T:=\Th(\M^*)$ has quantifier elimination and thus there is a quantifier-free $\La^*$-formula $\delta(\mybar{y})$ such that $\forall\mybar{y}\,(\chi(\mybar{y})\leftrightarrow\delta(\mybar{y}))\in T$. Thus by compactness there is an $\La^*$-sentence $\tau\in T$ such that
\begin{equation}\label{tauimplies}
\tau\models\forall\mybar{y}\,(\chi(\mybar{y})\leftrightarrow\delta(\mybar{y})).
\end{equation}
By \Cref{stillsmoothapprox} and the $\forall\exists$-axiomatisation of $T$ (see the proof of Proposition 5.4 in \cite{klm}), there exists $Q\in\Na$ such that ${\M_i}^*\models\tau$ for all $i>Q$. Now, consider some arbitrary $\mybar{c}\in {M_i}^m$ with $i>Q$. Since $\delta$ is quantifier-free and ${\M_i}^*\substruc\M^*$,
\[
\M^*\models\delta(\mybar{c}) \iff {\M_i}^*\models\delta(\mybar{c}).
\]
Hence by \eqref{tauimplies} we have
\[
\M^*\models\chi(\mybar{c}) \iff {\M_i}^*\models\chi(\mybar{c})
\]
because $\M^*\models\tau$ and ${\M_i}^*\models\tau$. But $\chi$ is an $\La$-formula and thus
\[
\M\models\chi(\mybar{c}) \iff \M_i\models\chi(\mybar{c}),
\]
as required.
\end{proof}

The following example shows that the converse of \Cref{sapproxmecshort} does not hold, in the following sense: An ultraproduct of an $R$-mec need not be elementarily equivalent to a smoothly approximable structure.

\begin{exmpl}
The class of all finite abelian groups is an $R$-mec (\Cref{abelian-groups}) and thus the subclass $\C$ of all finite cyclic groups of prime order is also an $R$-mec. Let $\mathcal{U}$ be a non-principal ultraproduct of $\C$. Then by {\L}os's theorem $\mathcal{U}$ is torsion-free. So $\mathcal{U}$ has infinitely 2-types:\ consider pairs $(x,x^k)$ for $k \in \Na$. Thus by the Ryll-Nardzewski Theorem $\mathcal{U}$ cannot be elementarily equivalent to an $\aleph_0$-categorical structure. Therefore, since smoothly approximable structures are $\aleph_0$-categorical, $\mathcal{U}$ cannot be elementarily equivalent to a smoothly approximable structure.
\end{exmpl}


\section{Lie coordinatisation}
\label{section:liecoord}

The goal of this section is to use Lie coordinatisation to prove the main result \Cref{dugaldsconjecturefull}, as conjectured by Macpherson. As such, our account of Lie coordinatisation is streamlined for this purpose and we leave some important notions from \cite{cherhrush} by the wayside, most notably orientation and orthogonality. That being said, we make explicit a number of details that are only implicit in \cite{cherhrush}, especially in our proofs of \Cref{thmsix} and \Cref{5.2.2}. Our presentation is based primarily on \cite{cherhrush}, with input from \cite{chs}.

The history of Lie coordinatisation does not lend itself to easy synopsis and we give only a very brief summary; see \secref{1} of \cite{cher} and \S\secref{1.1--1.2} of \cite{cherhrush} for a more detailed picture. The notion was developed by Cherlin and Hrushovski as (inter alia) an attempt to find a structure theory for smoothly approximable structures, building on the work of Kantor, Liebeck and Macpherson in \cite{klm}. Deep links between other model-theoretic notions were discovered through their investigation (\secref{1.2} of \cite{cherhrush}). In particular, it was shown that Lie coordinatisability and smooth approximation are equivalent (Theorem 2 in \cite{cherhrush}). Note that the classification of finite simple groups plays a fundamental role, albeit in the background.

In contrast to its mathematical depth, Lie coordinatisation has made only a shallow footprint in the literature, in part due to the development of simple theories. There are significant mathematical links between the two topics (see pp.~8--10 of \cite{cherhrush}), but simple theories have received more attention from model theorists. The reasons for this are manifold and a subject for debate, but I present two subjective opinions: Firstly, simple theories are quite simply easier to work with. The definition of a simple theory via the tree property is direct and without prerequisites (other than the usual background infrastructure of modern model theory), but as the reader will soon discover, the definition of Lie coordinatisability does not lend itself to swift comprehension and requires considerable technical machinery. Secondly, simple theories perhaps provide more mathematical relevancy. Simplicity is an important demarcation line among unstable theories and its study has led to a deeper understanding of independence. Moreover, simplicity theory has provided new insight into important non-$\aleph_0$-categorical theories such as ACFA and pseudofinite fields. Note however that Lie coordinatisation and the work of Cherlin, Lachlan and Harrington had a lot of implicit influence on the development of simplicity theory; early versions of \cite{hrushovski93} and \cite{cherhrush} significantly predate \cite{kimpillay}. Also note the discussion of the independence theorem on p.~9 of \cite{cherhrush}. (I thank Dugald Macpherson and Sylvy Anscombe for sharing their thoughts on the topic of this paragraph.)

The first publication on Lie coordinatisation was the paper \cite{hrushovski93} by Hrushovski, in joint work with Cherlin. Some technical issues were found in this paper (see p.~7 of \cite{cherhrush}) and corrected results were published in \cite{cher}, which is essentially an abridgement of the main text \cite{cherhrush}. The paper \cite{chs} by Chowdhury, Hart and Sokolovi\'{c} makes significant contributions and Hrushovski has published some further work on quasifiniteness in \cite{hrushovskiND}. There are also some unpublished notes \cite{hillsmart} by Hill and Smart. Lie coordinatisation arises in the context of asymptotic classes in \cite{elwesphd}, \cite{elwes}, \cite{macstein1} and \cite{macstein2}.

We now outline the structure of this section. In \Cref{Liebasicssec} we go over the basic concepts of Lie coordinatisation and in \Cref{Lieexamplessec} we provide two examples of Lie coordinatisable structures. \Cref{envelopessec} develops the notion of an envelope, which is fundamental to the rest of the section. We then move on to \Cref{macconjshortsec}, where we state and sketch a proof of a result (\Cref{thmsix}) that allows us to apply \Cref{sapproxmecshort} to obtain a short version of Macpherson's conjecture (\Cref{dugaldsconjectureshort}). \Cref{5.2.2sec} then provides us with the extra information needed to prove the full version of the conjecture in \Cref{sectionmacphersonconjfull}.


\subsection{Lie geometries and Lie coordinatisation}
\label{Liebasicssec}

We state the definition of Lie coordinatisation. We need to go over a number of preliminaries first, starting with Lie geometries. We refer the reader to chapter 7 of \cite{asch} for the terminology and theory of vector spaces with forms.

\begin{defn}[Linear Lie geometry, Definition 2.1.4 in \cite{cherhrush}]\label{liegeos} Let $K$ be a finite field. A \emph{linear Lie geometry} over $K$ is one of the following six kinds\,\footnote{In contrast to Definition 2.1.4 in \cite{cherhrush}, we use the word `kind' in order to avoid overuse of the word `type'.} of structures:
\begin{thmlist}[1.]
\item\emph{A degenerate space.} An infinite set in the language of equality.

\item\emph{A pure vector space.} An infinite-dimensional vector space $V$ over $K$ with no further structure.

\item\emph{A polar space.} Two infinite-dimensional vector spaces $V$ and $W$ over $K$ with a non-degenerate bilinear form $V\times W\to K$.

\item\emph{A symplectic space.} An infinite-dimensional vector space $V$ over $K$ with a symplectic bilinear form $V\times V\to K$.

\item\emph{A unitary space.} An infinite-dimensional vector space $V$ over $K$ with a unitary sesquilinear form $V\times V\to K$.

\item\emph{An orthogonal space.} An infinite-dimensional vector space $V$ over $K$ with a quadratic form $V\to K$ whose associated bilinear form is non-degenerate.
\end{thmlist}
\end{defn}

\begin{rem}We comment on \Cref{liegeos}.
\begin{thmlist}
\item We consider linear Lie geometries as two-sorted structures $(V,K)$, with a sort $V$ in the language of groups with an abelian group structure, a sort $K$ in the language of rings with a field structure, and a function $K\times V\to V$ for scalar multiplication. We call $V$ the \emph{vector sort} and $K$ the \emph{field sort}. (See pp.~5 and 12 of \cite{tentziegler} for a summary of multi-sorted structures and languages.) The elements of $K$ are named by constant symbols.\footnote{Note that this is what the prefix `basic' refers to in Definition 2.1.6 in \cite{cherhrush}. Since we always name the field elements by constant symbols, we suppress this prefix.} In the polar case, the vector sort is $V\cup W$ in the language of groups equipped with an equivalence relation with precisely two classes $V$ and $W$, each with an abelian group structure.
\item\label{ignorequad} We have ignored quadratic Lie geometries (Definition 2.1.4 in \cite{cherhrush}), as we do not need to consider them, save only to rule them out in the proof of \Cref{5.2.2}. They arise from the fact that in characteristic 2 every symplectic bilinear form has many associated quadratic forms.
\end{thmlist}
\end{rem}

\begin{lem}[Lemmas 2.2.8 and 2.3.19 in \cite{cherhrush}]\label{geoqe}
Every linear Lie geometry has quantifier elimination and is $\aleph_0$-categorical.
\end{lem}

\begin{defn}[Projective Lie geometry, Definition 2.1.7 in \cite{cherhrush}]\label{projgeo}
Let $L$ be a linear Lie geometry and let $\acl$ denote the usual model-theoretic algebraic closure in $L$. We define an equivalence relation $\sim$ on $L\setminus\acl(\0)$ as follows:
\[
a\sim b \text{ iff } \acl(a)=\acl(b)
\]
The \emph{projectivisation} of $L$ is defined to be the quotient structure arising from this equivalence relation:
\[
\left. ^{\textstyle L\setminus\acl(\0)} \middle/ _{\textstyle\sim} \right.
\]
A \emph{projective Lie geometry} is a structure that is the projectivisation of some linear Lie geometry.
\end{defn}

\begin{rem}[comment after Definition 2.1.7 in \cite{cherhrush}]\label{linearspan}
By quantifier elimination (\Cref{geoqe}), algebraic closure is just linear span and so a projective Lie geometry is a projective geometry in the usual sense.
\end{rem}

\begin{defn}[Affine Lie geometry, Definition 2.1.8 in \cite{cherhrush}]
An \emph{affine Lie geometry} is a structure of the form $(V,A,\oplus,-)$, where $V$ is the vector sort of a linear Lie geometry (but not a degenerate space), $A$ is a set, $\oplus\colon V\times A\to A$ is a regular group action and $-\colon A\times A\to V$ is such that $a=v\oplus b$ implies $a-b=v$. Here `regular' means that for every $a,b\in A$ there exists a unique $v\in V$ such that $a=v\oplus b$. In the polar case the structure is $(V,W,A,\oplus,-)$, where $\oplus\colon V\times A\to A$ is a regular group action and $-\colon A\times A\to V$ is such that $a=v\oplus b$ implies $a-b=v$.
\end{defn}

\begin{defn}[Lie geometry]
Linear, projective and affine Lie geometries are referred to collectively as \emph{Lie geometries}.
\end{defn}

The notions of canonical and stable embeddedness are fundamental to Lie coordinatisation:

\begin{defn}[Embedded structures, Definition 2.1.9 in \cite{cherhrush}]
Consider an $\La$-structure $\N$ and an $\La'$-structure $\M$ such that the underlying set $M$ is an $\La_N$-definable subset of $N$. Let $c\in\N\eq$ be a canonical parameter for $M$. (See \secref{8.2} of \cite{marker} or \secref{8.4} of \cite{tentziegler} for an introduction to canonical parameters.)
\begin{thmlist}
\item $\M$ is \emph{canonically embedded} in $\N$ if the $\La'_\0$-definable relations of $\M$ are precisely the $\La_c$-definable relations on $\M$; that is, for every $n\in\Na^+$, a subset $D\subseteq M^n$ is $\La'_\0$-definable in the structure $\M$ if and only if it is $\La_c$-definable in the structure $\N$. (The notation $\La'_\0$ isn't strictly necessary, since $\La'=\La'_\0$, but the subscript $\0$ is added to emphasise $\0$-definability.)

\item $\M$ is \emph{stably embedded} in $\N$ if every $\La_N$-definable relation on $\M$ is $\La_M$-definable in a uniform way; that is, for every $\La$-formula $\vphi(\mybar{x},\mybar{y})$, where $n:=l(\mybar{x}) \geq 1$ and $m:=l(\mybar{y})$, if $\vphi(\N^n,\mybar{a})\subseteq M^n$ for every $\mybar{a}\in N^m$, then there exists an $\La$-formula $\vphi'(\mybar{x},\mybar{z})$, where $r:=l(\mybar{z})$, such that for every $\mybar{a}\in N^m$ there exists $\mybar{a}'\in M^r$ such that $\vphi(\N^n,\mybar{a})=\vphi'(\N^n,\mybar{a}')$. (Note that we need not have $m=r$.)

\item $\M$ is \emph{fully embedded} in $\N$ if $\M$ is both canonically and stably embedded in $\N$.
\end{thmlist}
\end{defn}

Intuitively, $\M$ is fully embedded in $\N$ if $\N$ cannot place any additional structure on $\M$.

We won't need the following definition until \Cref{5.2.2sec}, but it follows on from the previous definitions.

\begin{defn}[Localisation, Definition 2.4.9 in \cite{cherhrush}]\label{localisation}
Let $P$ be a projective Lie geometry, arising from a linear Lie geometry $L$. Suppose that $P$ is fully embedded in an $\La$-structure $\M$. The \emph{localisation} of $P$ over a finite set $A\subset M$ is defined as follows: Let $f$ be the bilinear/sesquilinear form on $L$, where for a degenerate space or a pure vector space we define $f(v,w):=0$ for all $v,w\in L$ and for an orthogonal space $f$ is the bilinear form associated to the quadratic form on $L$. Define
\[
L_A^\perp:=\{v\in L:\text{$f(v,w)=0$ for all $w\in \acl(A)\cap L$}\}
\]
or, in the polar case,
\begin{align*}
L_A^\perp:=\{v\in V:&\,\text{$f(v,w)=0$ for all $w\in \acl(A)\cap W$}\} \\
&\cup\{v\in W:\text{$f(v,w)=0$ for all $w\in \acl(A)\cap V$}\}.
\end{align*}
Let $L_A^\perp /(L_A^\perp\cap\acl(A))$ be the quotient space, in the usual sense of a quotient of abelian groups. (This makes sense by \Cref{linearspan}.) Then the localisation of $P$ over $A$ is defined to be the projectivisation of $L_A^\perp /(L_A^\perp\cap\acl(A))$; that is, let $\sim$ be as in \Cref{projgeo} and then quotient $L_A^\perp /(L_A^\perp\cap\acl(A))$ by $\sim$.

We denote the localisation of $P$ over $A$ by $P/A$.
\end{defn} 

We are now ready to state the definition of Lie coordinatisation itself:

\begin{defn}[Lie coordinatisation, Definition 2.1.10 in \cite{cherhrush}]\label{Lie-isation-defn}
Let $\M$ be an $\La$-structure. A \emph{Lie coordinatisation} of $\M$ is an $\La_{\0}$-definable partial order $<$ of $M$ that forms a tree of finite height with an $\La_{\0}$-definable root $w$ such that the following condition holds: For every $a\in M\setminus\{w\}$ either the immediate predecessor $u$ of $a$ has only finitely many immediate successors (which implies $a\in\acl(u)$) or, if $a\not\in\acl(u)$, then there exist $b<a$ and an $\La_b$-definable projective Lie geometry $J$ fully embedded in $\M$ such that either
\begin{thmlist}
\item $a\in J$ or, if $a\not\in J$, then

\item there exist $c\in M$ with $b<c<a$ and an $\La_c$-definable affine Lie geometry $(V,A)$ fully embedded in $\M$ such that $a\in A$, the projectivisation of $V$ is $J$, and $J<V<A$,
\end{thmlist}
where for subsets $X,Y\subset M$ the notation $X<Y$ means that every element of $X$ lies in a lower level of the tree than every element of $Y$. We say that the Lie geometries $J$ and $(V,A)$ in the tree \emph{coordinatise} $\M$ and refer to them as \emph{coordinatising Lie geometries} (or just \emph{coordinatising geometries}). By a \emph{Lie coordinatised structure} we mean a structure equipped with a Lie coordinatisation.
\end{defn}

\begin{defn}[Lie coordinatisability, Definition 2.1.12 in \cite{cherhrush}]\label{Lie-isable-defn}
An $\La$-structure $\M$ is \emph{Lie coordinatisable} if it is $\0$-bi-interpretable (see \secref{2.5} of \cite{amsw} or \secref{2.4} of \cite{wolfphd}) with a Lie coordinatised structure that has finitely many $1$-types over $\0$.
\end{defn}

\begin{rem}\label{weakLie} We have actually defined so-called `weak Lie coordinatisability' (p.~17 of \cite{cherhrush}), since in \Cref{Lie-isation-defn} we did not stipulate the orientation condition relating to quadratic coordinatising geometries (Definition 2.1.10 in \cite{cherhrush}). This condition is important and cannot be ignored in general, but we can ignore it because we do not need to consider quadratic Lie geometries (\Cref{ignorequad}). For brevity we thus suppress the prefix `weak', the proof of \Cref{thmsix} being an exception. Note that the orientation condition is also ignored in \cite{chs} for the same reason (p. 517 of \cite{chs}).
\end{rem}

\begin{rem}\label{addingsorts}
In general it is important to maintain the distinction between Lie coordinatisation and Lie coordinatisability, but we freely move from the latter to the former by adding finitely many sorts from $\M\eq$ to $\M$.
\end{rem}

We quote two important results from \cite{cherhrush}:

\begin{lem}[Lemma 2.3.19 in \cite{cherhrush}]\label{liecat}
If $\M$ is Lie coordinatisable, then $\M$ is $\aleph_0$-categorical.
\end{lem}

\begin{thm}[Theorem 2 in \cite{cherhrush}]\label{smoothLieequiv}
Let $\M$ be an $\La$-structure. Then $\M$ is Lie coordinatisable if and only if $\M$ is smoothly approximable.
\end{thm}


\subsection{Examples}
\label{Lieexamplessec}

We give two examples of Lie coordinatisable structures, returning to \Cref{smoothapproxexmpl1,smoothapproxexmpl2}, which by \Cref{smoothLieequiv} we know must be Lie coordinatisable.

\begin{figure}[h!]
\centering
\def\svgwidth{180pt}
\small{
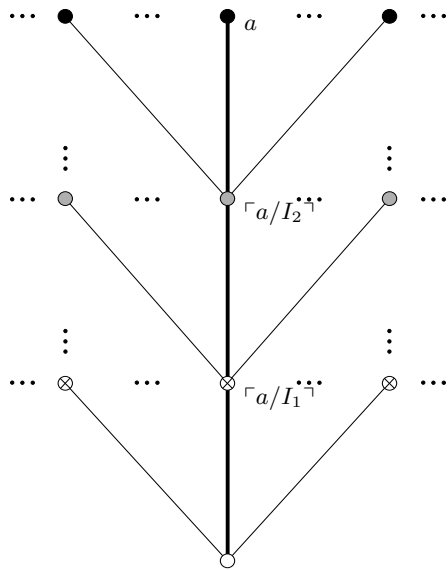}
\caption{A finite fragment of the tree from \Cref{coordinatisationexample1}, with the branch leading to the element $a$ in bold. The nodes are shaded according to membership: The white node is $\Gnum{\M}$, the crossed nodes are elements of $\M/I_1$, the grey nodes are elements of $(a/I_1)/I_2$, and the black nodes are elements of $a/I_2$. The small dots represent the rest of the tree.}\label{tree1}
\end{figure}

\begin{exmpl}[continuation of \Cref{smoothapproxexmpl1}]\label{coordinatisationexample1}
Consider a language $\La:=\{I_1,I_2\}$, where $I_1$ and $I_2$ are binary relation symbols. Let $\M$ be a countable $\La$-structure where $I_1^\M$ and $I_2^\M$ are equivalence relations such that $I_1^\M$ has infinitely many classes, $I_2^\M$ refines $I_1^\M$, every $I_1$-equivalence class contains infinitely many $I_2$-equivalence classes, and every $I_2$-equivalence class is infinite; that is, $\M$ is partitioned into infinitely many $I_1$-equivalence classes, each of which is then partitioned into infinitely many $I_2$-equivalence classes, each of which is infinite. We claim that $\M$ is Lie coordinatisable.

We first outline the tree structure. At the root we place $\Gnum{\M}$ (the canonical parameter of $\M$ in $\M\eq$, which is $\0$-definable), above which we place the $I_1$-classes, as imaginary elements of $\M\eq$. Above each $I_1$-class we then place the $I_2$-classes, again as imaginary elements of $\M\eq$, with every $I_2$-class above the $I_1$-class in which the $I_2$-class is contained. Finally, above each $I_2$-class we place the elements of $\M$ contained in that $I_2$-class. So this tree has height 3 and infinite width at each level.

Let's explain the notation used in \Cref{tree1}. So consider some arbitrary $a\in M$. For $j=1$ or 2, let $a/I_j$ denote the $I_j$-class that contains $a$ and let $\Gnum{a/I_j}$ denote the same $I_j$-class but as a member of $\M\eq$; so $\Gnum{a/I_j}\in\M\eq$ is a canonical parameter for the $a$-definable subset $a/I_j\subset M$. We define $(a/I_1)/I_2$ and $\Gnum{(a/I_1)/I_2}$ similarly.

We now use this notation to check that \Cref{Lie-isation-defn} holds for the tree. The imaginary element $\Gnum{a/I_1}$ lies in the $\Gnum{\M}$-definable degenerate projective geometry $\M/I_1$ and $\Gnum{\M}<\Gnum{a/I_1}$. The imaginary element $\Gnum{a/I_2}$ lies in the $\Gnum{a/I_1}$-definable degenerate projective geometry $(a/I_1)/I_2$ and $\Gnum{a/I_1}<\Gnum{a/I_2}$. Finally, the real element $a$ lies in the $\Gnum{a/I_2}$-definable degenerate projective geometry $a/I_2$ and $\Gnum{a/I_2}<a$. Adjoining a finite number of sorts from $\M\eq$ (recall \Cref{addingsorts}), each of these geometries is fully embedded in $\M$. (Note that $\M/I_2$ is \emph{not} fully embedded, since $I_1$ defines extra structure on $\M/I_2$ that is not definable within $\M/I_2$ using equality alone.) So $\M$ is indeed Lie coordinatisable.
\end{exmpl}

\begin{rem}\label{coordinatisationexample1rem}
This example generalises to the case where we have $n$ equivalence relations $I_1,\ldots,I_n$ such that there are infinitely many $I_1$-classes, $I_{j+1}$ refines $I_j$ and every $I_j$-class contains infinitely many $I_{j+1}$-classes (for $1\leq j\leq n-1$), and every $I_n$-class is infinite. At the base of the tree (the \nth{0} level) we place $\Gnum{\M}$. At the $j^\mathrm{th}$ level (for $1\leq j\leq n-1$) we place the $I_j$-classes, as imaginary elements of $\M\eq$, with every $I_j$-class above the $I_{j-1}$-class in which the $I_j$-class is contained. Finally, at the top of the tree (the $n^\mathrm{th}$ level) we place the elements of $M$, with each $a\in M$ placed above $\Gnum{a/I_n}$.
\end{rem}

\begin{figure}[h!]
\centering
\def\svgwidth{180pt}
\small{
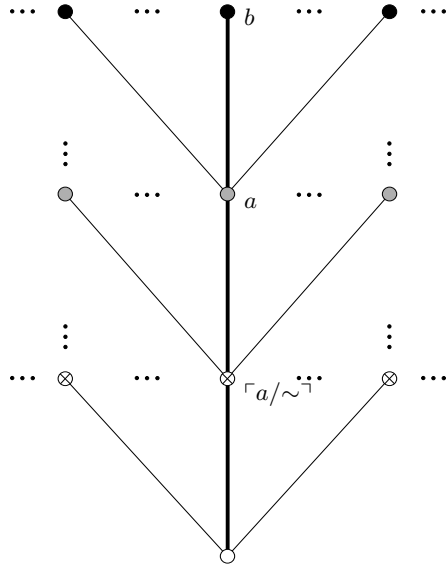}
\caption{A finite fragment of the tree from \Cref{coordinatisationexample2}, with the branch leading to the element $b\in\M_a$ in bold. The nodes are shaded according to membership: The white node is the zero vector, the crossed nodes are elements of $P(\M_0)$, the grey nodes are elements of $a/{\sim}$, and the black nodes are elements of $\M_a$. Note that there are only finitely many (in fact $p-1$) nodes immediately above each crossed node. The small dots represent the rest of the tree.}\label{tree2}
\end{figure}

\begin{exmpl}[Example 2.1.11 in \cite{cherhrush}; continuation of \Cref{smoothapproxexmpl2}]\label{coordinatisationexample2}
Let \allowbreak $\La := \{0,+\}$ and let $p$ be a fixed prime number. (The case $p=2$ \emph{is} allowed.) We define $M$ to be the direct sum of $\omega$-many copies of $\Z/p^2\Z$, i.e.\
\[
M:=\{(a_i)_{i<\omega}:a_i\in\Z/p^2\Z\,\text{ and $a_i=0$ for all but finitely many $i$}\}.
\]
(We specify the direct sum because it is countable, unlike the direct product.) The set $M$ naturally forms an $\La$-structure $\M$, the $\La$-structure arising component-wise from the $\La$-structure of the group $\Z/p^2\Z$. Explicitly:\ $0^\M:=(0)_{i<\omega}$ and $(a_i)_{i<\omega} + (b_i)_{i<\omega}:=(a_i+b_i)_{i<\omega}$. For brevity we write 0 for $0^\M$. We claim that $\M$ is Lie coordinatisable.

We first introduce some notation: For $v\in M$ let $\M_v:=\{a\in M:pa=v\}$, where $pa:=\underbrace{a+a+\cdots+a}_{\text{$p$ times}}$. Observe that $\M_0$ has a vector space structure over $\F_p$ and thus is a linear Lie geometry over $\F_p$. Let $P(\M_0)$ be the projectivisation of $\M_0$ (\Cref{projgeo}). Then $P(\M_0)=(\M_0\setminus\{0\})/{\sim}$, where $a\sim b$ if and only if $a=rb$ for some $r\in\F_p$ (recall \Cref{linearspan}). So $|a/{\sim}|=p-1$ for all $a\in\M_0$. Adjoining a sort for $P(\M_0)$ (recall \Cref{addingsorts}), we also have that $P(\M_0)$ is fully embedded in $\M$.

We now outline the tree structure. At the root we place $0$, above which we place the elements of $P(\M_0)$, considered as imaginary elements of $\M\eq$. On the next level we place the elements of $\M_0\setminus\{0\}$, with each $a$ placed above $\Gnum{a/{\sim}}$. Finally, the top level contains the elements of $\M\setminus\M_0$, with each $b\in\M_a$ placed above $a$. So we have a tree of height 3 and infinite width at each level, although the second level comprises an infinite amount of finite branching. Note that we're using the fact here that if $b\in\M\setminus\M_0$, then $b\in\M_a$ for some $a\in\M_0$. The proof of this fact is straightforward: Suppose that $b\in\M\setminus\M_0$. Then $pb\neq 0$. So $pb=a$ for some $a\in M$. Then $pa=p(pb)=p^2b=0$, since $p^2c=0$ for all $c\in M$. So $b\in\M_a$ and $a\in\M_0$, as required.

Let's check that \Cref{Lie-isation-defn} holds for this tree. So consider some arbitrary non-zero $a\in\M_0$ and $b\in\M_a$. See \Cref{tree2} for an illustration. The imaginary element $\Gnum{a/{\sim}}$ lies in the $0$-definable projective geometry $P(\M_0)$, which is fully embedded, as noted in the previous paragraph, and $0<\Gnum{a/{\sim}}$. The real element $a$ is algebraic over $\Gnum{a/{\sim}}$, since $a/{\sim}$ is $\Gnum{a/{\sim}}$-definable and finite, again as noted in the previous paragraph, and $\Gnum{a/{\sim}}<a$. This leaves us with the top level of the tree, which we deal with in the next paragraph.

Firstly, observe that $0<a<b$. The real element $a$ defines an affine geometry $(\M_0,\M_a)$, where $\M_0$ is the $\F_p$-vector space, $\M_a$ is the $\M_0$-affine space, and the action $\M_0\times\M_a\to\M_a$ is given by $(u,v)\mapsto u+v$. (This action is well-defined, since $p(u+v)=pu+pv=0+a=a$ and so $u+v\in\M_a$.) As we have already noted, the projectivisation of $\M_0$ is $P(\M_0)$, which is a fully embedded, $0$-definable projective geometry, and we have $b\in\M_a$ by assumption. So the tree structure does indeed satisfy the definition of Lie coordinatisation.
\end{exmpl}

\begin{rem}\label{coordinatisationexample2rem}
This example generalises to the direct sum of $\omega$-many copies of $\Z/p^n\Z$, for any $n\in\Na^+$. When $n=1$, the tree structure is the same as in the case $n=2$, except that $\M_0\setminus\{0\}$ forms the top level, since $\M\setminus\M_0=\0$. When $n\geq 3$, the first three levels ($0$, $P(\M_0)$ and $\M_0$) are the same, but at the third level one places the elements of $\{b\in M: \text{$b\in\M_a$ for some $a\in\M_0$}\}$, instead of simply $\M\setminus\M_0$, and at the $(j+1)^\mathrm{th}$ level (for $1\leq j\leq n$) one places $\{c\in M: \text{$c\in\M_b$ for some $b$ in the $j^\mathrm{th}$ level}\}$. The $(n+1)^\mathrm{th}$ level is the upper-most level.
\end{rem}

\begin{rem}
In both \Cref{coordinatisationexample1} and \Cref{coordinatisationexample2} the tree is nicely stratified, namely root--degenerate--degenerate--degenerate in the former and root--projective--algebraic--affine in the latter. This need not be the case, however: There are Lie coordinatising trees containing maximal chains of different lengths. For example, one could take the disjoint union (in a suitable language, with a common root) of two Lie coordinatising trees of different heights.
\end{rem}


\subsection{Standard systems of geometries and envelopes}
\label{envelopessec}

We develop the key notion of an envelope of a Lie coordinatised structure. Our presentation is a simplified version of that given in \cite{cherhrush}, streamlined for the purpose of stating and proving \Cref{5.2.2}. We begin with the notion of a standard system of geometries:

\begin{defn}[Standard system of geometries, Definitions 2.5.1 and 2.5.6 in \cite{cherhrush}]\label{stnsysdefn}
Let $\M$ be a Lie coordinatised $\La$-structure. A \emph{standard system of geometries} in $\M$ is a $\0$-definable function $J\colon A\to\M\eq$ whose domain $A$ is the set of realisations of a 1-type over $\0$ in $\M$ (or the canonical parameter thereof) and whose image is a set of canonical parameters of coordinatising projective Lie geometries of the same kind, i.e.\ $J(a)$ and $J(b)$ are isomorphic for every $a,b\in A$, such that each $a \in A$ parametrises its image $J(a)$.

By `$\0$-definable' we mean that there exists an $\La$-formula $\vphi(x,y)$ such that $\vphi(\M,a)=J(a)$ for every $a\in A$. We write $\dom(J)$ for the domain $A$ of $J$. We abbreviate the term `standard system of geometries' as `SSG' and its plural `standard systems of geometries' as `SSGs'.

If two SSGs have the same image, then they are equivalent. (This is a simplication of the notion of orthogonality developed in \cite{cherhrush}, which we purposefully circumvent in the present work in order to avoid unnecessary complexity.)
\end{defn}

\begin{defn}[Approximations and dimension functions, Definition 3.1.1 in \cite{cherhrush}]\label{liedimfuncdefn}~
\begin{thmlist}
\item A Lie geometry is by definition required to be infinite-dimensional. If we change this to finite-dimensional, then we have an \emph{approximation} of a Lie geometry. For example, if $J$ is a degenerate space, then an approximation of $J$ is a finite set in the language of equality, or if $J$ is the projectivisation of an infinite-dimensional pure vector space over a finite field $K$, then an approximation of $J$ is the projectivisation of a finite-dimensional pure vector space over $K$.

\item Let $\M$ be a Lie coordinatised structure. A \emph{dimension function} is a function $\mu$ on a finite set $S$ of non-equivalent SSGs in $\M$ that assigns an approximation to each $J\in S$, i.e.\ $\mu(J)$ is an approximation of $J(a)$ for some $a\in\dom(J)$; note that this is independent of the choice of $a$, since $J(a)$ is by definition the same kind of projective Lie geometry for every $a\in\dom(J)$. We call $S$ the \emph{domain} of $\mu$, which we denote by $\dom(\mu)$. For each $J\in\dom(\mu)$ we define $\dim \mu(J)$ to be the dimension of $\mu(J)$, except in the degenerate case, where we instead define $\dim\mu(J):=|\mu(J)|$.
\end{thmlist}
\end{defn}

\begin{defn}[$\mu$-Envelope, Definition 3.1.1 in \cite{cherhrush}]\label{envelopedefn} Let $\M$ be a Lie coordinatised structure. Then a \emph{$\mu$-envelope} is a pair $(E,\mu)$ consisting of a finite subset $E\subset M$ and a dimension function $\mu$ for which the following three conditions holds:
\begin{thmlist}
\item $E$ is algebraically closed in $\M$. (Note that this implies that $E$ is a substructure of $\M$.)
\item For every $a\in M\setminus E$ there exist $J\in\dom(\mu)$ and $b\in\dom(J)\cap E$ such that $\acl(E)\cap J(b)$ is a proper subset of $\acl(E,a)\cap J(b)$.
\item For every $J\in\dom(\mu)$ and for any $b\in\dom(J)\cap E$, $J(b)\cap E$ and $\mu(J)$ are isomorphic.
\end{thmlist}
\end{defn}

\begin{rem}~
\begin{thmlist}
\item We often denote a $\mu$-envelope by $E$, rather than $(E,\mu)$, leaving the dimension function as implicit. We similarly often use the term `envelope', rather than `$\mu$-envelope'.

\item It may help the reader's intuition to know that envelopes form homogeneous substructures of $\M$ (Lemma in 3.2.4 \cite{cherhrush}). Indeed, this is how the left-to-right direction of \Cref{smoothLieequiv} is proved (pp.~61--62 of \cite{cherhrush}).

\item In general one can have countably infinite approximations and envelopes, but we do not need to consider them.
\end{thmlist}
\end{rem}

The following definition is fundamental to the work in \Cref{5.2.2sec}:

\begin{defn}[Definition 3.1.1, Notation 5.2.1 and Proposition 5.2.2 in \cite{cherhrush}]\label{envelopedims}
Let $\M$ be a Lie coordinatised structure and consider a $\mu$-envelope $(E,\mu)$ in $\M$, where $\dom(\mu)=\{J_1,\ldots,J_s\}$. For each $J_i$ we define $d_E(J_i):=\dim\mu(J_i)$. We further define $d^*_E(J_i):=(-\sqrt{q})^{d_E(J_i)}$, where $q$ is the size of the base finite field of $\mu(J_i)$, or $d^*_E(J_i):=d_E(J_i)$ in the degenerate case. (Taking $-\sqrt{q}$, rather than just $q$, does initially look strange. It is done solely for unitary spaces:\ see \hyperlink{unitarycalcs}{the end of the proof} of \Cref{5.2.2}.) Finally, we define $\mybar{d^*}(E):=(d_E^*(J_1),\ldots,d_E^*(J_s))$.
\end{defn}

We illustrate the preceding definitions by returning to \Cref{coordinatisationexample1,coordinatisationexample2}:

\begin{exmpl}[continuation of \Cref{coordinatisationexample1}]
Recall that $\M$ is partitioned into infinitely many $I_1$-equivalence classes, each of which is then partitioned into infinitely many $I_2$-equivalence classes, each of which is infinite.

Put simply, an example of an envelope in this case is a subset $E\subseteq\M$ that intersects a fixed number ($n_1$) of $I_1$-classes, a fixed number ($n_2$) of $I_2$-classes within each of these $I_1$-classes, and a fixed number ($n_3$) of elements within each of these $I_2$-classes. So, up to $\La$-isomorphism, an envelope is given by a triple $(n_1,n_2,n_3)$. Two examples of envelopes are
\begin{align*}
E_1&:=\{a_{ijk}:1\leq i\leq 3,1\leq j\leq 6,1\leq k\leq 1\}\\
\text{and }\,E_2&:=\{a_{ijk}:19\leq i\leq 21,3\leq j\leq 8,2015\leq k\leq 2015\},
\end{align*}
where we use the enumerations from \Cref{smoothapproxexmpl1}. The triple for both $E_1$ and $E_2$ is $(n_1,n_2,n_3)=(3,6,1)$. Let's now explain this in terms of SSGs and dimension functions.

\pagebreak[2]
Consider the following three SSGs in $\M$:

\vspace{.25\baselineskip}

\begin{enumerate}[(i)]
\item $J_\alpha\colon\{\Gnum{\M}\}\to\M\eq$, where $J_\alpha(\Gnum{\M}):=\Gnum{\M/I_1}$;
\item $J_\beta\colon\M\to\M\eq$, where $J_\beta(a):=\Gnum{(a/I_1)/I_2}$; and
\item $J_\gamma\colon\M\to\M\eq$, where $J_\gamma(a):=\Gnum{a/I_2}$.
\end{enumerate}

\vspace{.25\baselineskip}

\noindent These are in fact the only SSGs in $\M$, since $\Gnum{\M/I_1}$, $\Gnum{(a/I_1)/I_2}$ and $\Gnum{a/I_2}$ are the only kinds of coordinatising projective Lie geometries in the Lie coordinatisation of $\M$ and because there is only one 1-type over $\0$, its realisation being $\M$. A dimension function $\mu$ on $\{J_\alpha,J_\beta,J_\gamma\}$ assigns an approximation to each of $J_\alpha(\Gnum{\M})$, $J_\beta(a)$ and $J_\gamma(a)$, where $a$ is arbitrary. An approximation of a given Lie geometry is determined by the dimension of the approximation, which in this case is equal to the size of the approximation, since all the projective Lie geometries are degenerate. Thus $\mu$ is determined by a choice of triple $(n_1,n_2,n_3)$. So, if $\mu$ is given by a triple $(n_1,n_2,n_3)$, then a $\mu$-envelope $E$ is a choice of $n_1$ $I_1$-classes, of $n_2$ $I_2$-classes within each of the chosen $I_1$-classes and finally of $n_3$ elements within each of the chosen $I_2$-classes. Furthermore, again because all the projective Lie geometries in this example are degenerate, we have $\mybar{d^*}(E)=(n_1,n_2,n_3)$.
\end{exmpl}

\begin{exmpl}[continuation of \Cref{coordinatisationexample2}]
Recall that $\M$ is a direct sum of $\omega$-many copies of $\Z/p^2\Z$. We first find the SSGs in $\M$. Since there is only one coordinatising projective Lie geometry in the Lie coordinatisation of $\M$, namely $P(\M_0)$, there is only one possible image for an SSG in $\M$, namely $\{\Gnum{P(\M_0)}\}$. Thus there is only one SSG in $\M$ up to equivalence. An example is the following:
\[
J\colon\{ 0\}\to\M\eq, \text{ where } J(0):=\Gnum{P(\M_0)}.
\]

We now consider dimension functions. A dimension function $\mu$ on $\{J\}$ assigns an approximation to $J(0)$. An approximation of $P(\M_0)$ is a finite-dimensional subspace of $P(\M_0)$, which is determined by its dimension $n$ (since the base field $\F_p$ is fixed). Thus, since $J(0)=\Gnum{P(\M_0)}$, $\mu$ is determined by $n$. A $\mu$-envelope $E$ is then a particular choice of an $n$-dimensional subspace of $P(\M_0)$. Such a subspace is a finite power of $\Z/p^2\Z$; that is, a subset
\[
\{(a_i)_{i<\omega}\in M:\text{$a_i\neq 0$ only if $i=t_j$ for some $j$}\}
\]
given by $n$ distinct integers $t_1,\ldots,t_n\in\Na^+$. Since the base field is $\F_p$, which has size $p$, we have
\[
\mybar{d^*}(E)=((-\sqrt{p})^n).
\]
(Here $\mybar{d^*}(E)$ is  1-tuple, hence the apparently superfluous brackets.)
\end{exmpl}


\subsection{Macpherson's conjecture, short version}
\label{macconjshortsec}

We now take a big step towards proving \Cref{dugaldsconjecturefull} by proving a shorter version, namely \Cref{dugaldsconjectureshort}, where the existence of a multidimensional exact class is asserted but the nature of the measuring functions is not specified. We first sketch a proof of part 2 of Theorem 6 from \cite{cherhrush}, as this result is crucial to our proof of \Cref{dugaldsconjectureshort}. The key ingredients needed to prove the result are contained in \cite{cherhrush}, namely Propositions 4.4.3, 4.5.1 and 8.3.2 and their proofs, but the (non-trivial) argument putting them together is not made completely explicit. We state the result in a way that is convenient for our present purposes, but it is essentially the same as the original statement in \cite{cherhrush}, the only significant difference being the use of the equivalence of Lie coordinatisation and smooth approximation (\Cref{smoothLieequiv}).

\begin{thm}\label{thmsix}
Let $\La$ be a finite language and let $d\in\Na^+$. Define $\C(\La,d)$ to be the class of all finite $\La$-structures with at most $d$ 4-types. Then there is a finite partition $\mathcal{F}_1,\ldots,\mathcal{F}_k$ of $\C(\La,d)$ such that the $\La$-structures in each $\mathcal{F}_i$ smoothly approximate an $\La$-structure $\mathcal{F}_i^*$. Moreover, the $\mathcal{F}_i$ are definably distinguishable: For each $\mathcal{F}_i$ there exists an $\La$-sentence $\chi_i$ such that for all $\M\in\C(\La,d)$ above some minimum size, $\M\models\chi_i$ if and only if $\M\in\mathcal{F}_i$.
\end{thm}

\begin{proof}[Sketch of proof]\hspace{-0.4em}\footnote{The main argument was given by Hrushovski in email correspondence and Macpherson provided essential input by working out key details. The contribution of the present author lay in working through further details and writing up the proof.} We first show that there cannot exist infinitely many pairwise elementarily inequivalent Lie coordinatisable $\La$-structures with the same skeletal type, where a skeletal type is, roughly speaking, a full description of the Lie coordinatising tree structure in an extended language $\La_{sk}$; see \secref{4.2} of \cite{cherhrush} for the full definition. So, for a contradiction, suppose that there are in fact infinitely many such $\La$-structures $\{\N_i:i<\omega\}$ with the same skeletal type $S$. Working in $\La_{sk}$, by a judicious choice of ultrafilter we can take a non-principal ultraproduct $\N^*$ of the $\N_i$ such that $\N^*\not\equiv\N_i$ for all $i<\omega$. We may assume that $\N^*$ is countable by moving to a countable elementary substructure. Since the skeletal type $S$ is expressible in $\La_{sk}$ (this is a general fact of skeletal types, not just $S$) and true in each $\N_i$, by {\L}os's theorem $\N^*$ is Lie coordinatised and has skeletal type $S$. Work in chapter 4 of \cite{cherhrush}, especially Proposition 4.4.3 and its proof, shows that every Lie coordinatised structure is quasifinitely axiomatised. Thus $\N^*$ is quasifinitely axiomatised, which means that $\Th(\N^*)$ is axiomatised by a sentence $\sigma$ and an axiom schema of infinity specifying that every dimension in each coordinatising Lie geometry of $\N^*$ is infinite, where we consider $\Th(\N^*)$ as an $\La'$-theory in a finite language $\La'$ containing $\La_{sk}$. This axiom schema of infinity holds for all the $\N_i$ because they each have the same skeletal type as $\N^*$. Furthermore, again by {\L}os's theorem, there exists $j<\omega$ such that $\N_j\models\sigma$. Therefore $\N^*\equiv\N_j$, a contradiction.

We now return to the original class $\C:=\C(\La,d)$. We take an infinite ultraproduct $\U^*$ of the structures in $\C$. We take this ultraproduct in a non-standard model of set theory, working with some suitable G\"{o}del coding of formulas, which allows us to consider $\U^*$ as an $\La^*$-structure, where $\La^*$ is the ultrapower of the language $\La$; that is, $\La^*$ extends $\La$ by including infinitary formulas with nonstandard G\"{o}del numbers, although the number of free variables in any given formula remains finite. We may again assume that $\U^*$ is countable by moving to a countable elementary substructure. $\U^*$ is $4$-quasifinite (Definition 2.1.1 in \cite{cherhrush}) and thus by Theorem 3 in \cite{cherhrush} is weakly Lie coordinatisable (see \Cref{weakLie}). So by Proposition 7.5.4 in \cite{cherhrush} the $\La$-reduct $\U$ of $\U^*$ is also weakly Lie coordinatisable. The $\La$-structure $\U$ thus has a skeletal type. By the first part of the proof there can be only finitely many pairwise elementarily inequivalent Lie coordinatisable $\La$-structures with this skeletal type, say $\mathcal{F}_1^*,\ldots,\mathcal{F}_k^*$. By Proposition 4.4.3 in \cite{cherhrush}, each $\mathcal{F}_i^*$ has a characteristic sentence, say $\chi_i$. The $\chi_i$ yield a partition $\C=\mathcal{F}_1\cup\ldots\cup\mathcal{F}_k$, where each $\chi_i$ is true in all $\M\in\mathcal{F}_i$ and false in all $\M\in\mathcal{F}_j$ for $j\neq i$, potentially with the exception of some small structures. Moreover, again by Proposition 4.4.3, this partition is such that each $\M\in\mathcal{F}_i$ is an envelope of $\mathcal{F}_i^*$ and so by work in chapter 3 of \cite{cherhrush} the structures in $\mathcal{F}_i$ smoothly approximate $\mathcal{F}_i^*$.

Note that the work cited from chapter 4 of \cite{cherhrush} is written in terms of Lie coordinatisability, but inspection of the proofs shows that weak Lie coordinatisability suffices (see \Cref{weakLie}).
\end{proof}

\begin{cor}[Macpherson's conjecture, short version]\label{dugaldsconjectureshort}
For any countable language $\La$ and any $d\in\Na^+$ there exists $R$ such that the class $\C(\La,d)$ of all finite $\La$-structures with at most $d$ 4-types is an $R$-mec in $\La$.
\end{cor}

\begin{proof} Let $\C:=\C(\La,d)$. The reader should recall \Cref{smallermac}, as we will use it at various points in this proof.

First suppose that $\La$ is finite. By \Cref{thmsix}, $\C$ can be finitely partitioned into subclasses $\mathcal{F}_1,\ldots,\mathcal{F}_k$ such that the structures in each $\mathcal{F}_i$ smoothly approximate an $\La$-structure $\mathcal{F}_i^*$. Thus by \Cref{sapproxmecshort} each $\mathcal{F}_i$ is an $R_i$-mec in $\La$ for some $R_i$. Let $R_\La:=R_1\cup\cdots\cup R_k$. We claim that $\C$ is an $R_\La$-mec in $\La$.

We prove this claim: Let $\vphi(\mybar{x},\mybar{y})$ be an $\La$-formula with $n:=l(\mybar{x}) \geq 1$ and $m:=l(\mybar{y})$. Since each $\mathcal{F}_i$ is an $R_i$-mec, we have a suitable finite partition $\pt_i$ of each $\mathcal{F}_i(m)$. Then $\pt_1\cup\dots\cup\pt_k$ is a finite partition of $\C(m)$ and so $\C$ is a weak $R_\La$-mec in $\La$. It remains to show that the definability clause holds. We again use \Cref{thmsix}: For each $\mathcal{F}_i$ there is an $\La$-sentence $\chi_i$ such that $\M\models\chi_i$ if and only if $\M\in\mathcal{F}_i$, for sufficiently large $\M$. So, by conjoining $\chi_i$ to the defining $\La$-formulas of each $\pt_i$, we satisfy the definability clause, using \Cref{finiteexceptions} to deal with the finite number of potential exceptions. So the claim is proved.

Now suppose that $\La$ is infinite. Consider some arbitrary finite $\La'\subset\La$ and let $\C_{\La'}$ denote the class of all $\La'$-reducts of structures in $\C$. Each structure in $\C_{\La'}$ has at most $d$ 4-types, since a reduct cannot have more types than the original structure. Thus, by the first part of the proof, $\C_{\La'}$ is an $R_{\La'}$-mec in $\La'$. (It could be the case that $\C_{\La'}$ is a proper subclass of the class of all finite $\La'$-structures with at most $d$ 4-types, but that wouldn't matter, since a subclass of an $R$-mec is also an $R$-mec.) Let $\mathbb{L}$ be the set of all finite subsets of $\La$ and define
\[
R:=\bigcup_{\La'\in\mathbb{L}} R_{\La'}.
\]
Then each $\C_{\La'}$ is an $R$-mec in $\La'$ by \Cref{smallermac}. Therefore $\C$ is an $R$-mec in $\La$ by \Cref{reductmec}.
\end{proof}

\begin{rem}\label{4-types}
The reader may well be wondering what's so special about $4$-types. Well, firstly, if there is a bound on the number of $n$-types, then there is a bound on the number of $k$-types for all $k\leq n$. So in the statement of \Cref{dugaldsconjecturefull} we could replace $4$-types with $n$-types for any $n>4$ and the result would still go through. As for $4$ itself, the explanation goes deeper and we will not go into detail. However, put \emph{very} roughly, the number 4 arises because the projective linear group preserves the cross-ratio, which is a projective invariant on $4$-tuples of colinear points. The classification of finite simple groups also plays a role. Details can be found in \secref{6} of \cite{amsw}, \cite{klm} and \cite{macsmooth}. Note that in \cite{macsmooth} the original bound on $5$-types, as given in \cite{klm}, is improved to one on $4$-types.
\end{rem}


\subsection{Definable sets in envelopes}
\label{5.2.2sec}

\Cref{dugaldsconjectureshort} provides no information about the structure of $R$, only its existence. In this section we use Lie geometries to ascertain information about the nature of $R$. We first need to define a rank, which we name \emph{CH-rank} after Cherlin and Hrushovski:

\begin{defn}[CH-rank, Definition 2.2.1 in \cite{cherhrush}]\label{rank}
Let $\M$ be an $\La$-structure and let $D\subseteq\M\eq$ be a parameter-definable set. We define the \emph{CH-rank} of $D$ as follows:
\begin{thmlist}
\item $\rk(D)=-1$ if and only if $D=\0$.
\item $\rk(D)>0$ if and only if $D$ is infinite.
\item For $n\in\Na$, $\rk(D)\geq n+1$ if and only if there exist parameter-definable subsets $D_1,D_2\subseteq\M\eq$ and parameter-definable functions $\pi\colon D_1\to D$ and $f\colon D_1\to D_2$ such that:

\vspace{.25\baselineskip}

\begin{enumerate}[(a)]
\item $\rk(\pi^{-1}(d))=0$ for all $d\in D$;
\item $\rk(D_2)>0$; and
\item $\rk(f^{-1}(d))\geq n$ for all $d\in D_2$.
\end{enumerate}

\vspace{.25\baselineskip}
\end{thmlist}
If $\rk(D)>n$ for all $n\in\Na$, then we define $\rk(D)=\infty$.

Note that we will often drop the prefix `CH-' and simply refer to `rank'. 
\end{defn}

\begin{rem}
The core idea of the preceding definition is straightforward: A set has rank at least $n+1$ iff it can be parameter-definably partitioned into infinitely many subsets of rank at least $n$. The role of $\pi$ in the definition is to preserve rank under finite parameter-definable projections; however, this will not be necessary for the pruposes of the present work and thus we will assume throughout that $D_1=D$ and $\pi = \Id$, where $\Id$ denote the identity function.
\end{rem}

We provide examples of CH-rank by returning to our running examples:

\begin{exmpl}[continuation of \Cref{coordinatisationexample1}]
Recall that $\M$ is partitioned into infinitely many $I_1$-equivalence classes, each of which is then partitioned into infinitely many $I_2$-equivalence classes, each of which is infinite. Also recall that for $a\in M$ and $j=1$ or 2, $a/I_j$ denotes the $I_j$-class containing $a$ while $\Gnum{a/I_j}$ denotes the same $I_j$-class but as a member of $\M\eq$; in other words $\Gnum{a/I_j}\in\M\eq$ is a canonical parameter for the $a$-definable subset $a/I_j\subset M$.

We first calculate the rank of an $I_2$-class. So let $D:=a/I_2$ for some $a\in M$. Thus, since $D$ is infinite, $\rk(D)\geq 1$. Moreover, we see that $\rk(D) \ngeq 1+1$, for otherwise we would be able to parameter-definably partition an $I_2$-class into infinitely many infinite subsets, which is not possible in the $\La$-structure of $\M$. So the rank of an $I_2$-class is 1.

We now calculate the rank of an $I_1$-class. So let $D:=a/I_1$ for some $a\in M$. We set $D_2:=\{\Gnum{b/I_2}:b\in D_1\}$ and $f(b):=\Gnum{b/I_2}$. Since $D_2$ is infinite and each $I_2$-class has rank 1 (as shown in the previous paragraph), we see that $\rk(D) \geq 1+1$. We see that $\rk(D) \ngeq 2+1$ by the same reasoning given in the previous paragraph: In the $\La$-structure of $\M$ the only parameter-definable infinite partition of an $I_1$-class into infinite subsets is the partition induced by $I_2$; there is no other parameter-definable infinite partition of an $I_1$-class and there is no parameter-definable way to refine $I_2$. So the rank of an $I_1$-class is 2.

Lastly, we show that $\M$ has rank 3. So let $D:=M$. We set $D_2:=\{\Gnum{a/I_1}:a\in D\}$ and $f(b)=\Gnum{b/I_1}$. Since $D_2$ is infinite and each $I_1$-class has rank 2 (as already shown), we see that $\rk(D) \geq 2+1$. By similar reasoning given in the previous paragraphs, we have $\rk(D) \ngeq 3+1$. So $\M$ has rank 3.
\end{exmpl}

\begin{exmpl}[continuation of \Cref{coordinatisationexample2}]
Recall that $\M$ is a direct sum of $\omega$-many copies of $\Z/p^2\Z$. We define $\M_v:=\{a\in M:pa=v\}$ and $P(\M_0)=(\M_0\setminus\{0\})/{\sim}$, where $a\sim b$ if and only if $a=rb$ for some $r\in\F_p$ (recall \Cref{linearspan}). So $|a/{\sim}|=p-1$ for all $a\in\M_0$.

Consider some arbitrary $v \in \M$. Then $\rk(\M_v) \geq 1$ because $\M_v$ is infinite. Likewise $\rk(P(\M_0)) \geq 1$ because $P(\M_0)$ is also infinite. We further see that $\rk(\M_v) \ngeq 1+1$, since there is no parameter-definable infinite partition of $\M_v$ into infinite subsets. Similarly $\rk(P(\M_0)) \ngeq 1+1$. So $\rk(\M_v) = \rk(P(\M_0)) = 1$.

We now show that $\rk(\M\setminus \M_0)$ has rank 2. Set $D := \M\setminus \M_0$, $D_1 := D$, $\pi := \Id$, $D_2 := \M_0$ and $f(b) := pb$. Then for each $a \in D_2$ we have $f^{-1}(a) = \M_a$ and so $\rk(f^{-1}(a)) \geq 1$, since $\M_a$ is infinite. Thus $\rk(\M\setminus \M_0) \geq 1 + 1$. We see that $\rk(\M\setminus \M_0) \ngeq 2 + 1$ because there is no further parameter-definable partitioning. So $\rk(\M\setminus \M_0) = 2$.

Lastly, we see $\M$ has rank 2, since it is a superset of the rank-2 set $\M\setminus \M_0$ and because there is no further parameter-definable partitioning in $\M$.
\end{exmpl}

With a rank now defined, we are in a position to prove \Cref{5.2.2}. This result provides us with information about the sizes of definable sets in envelopes, which we will then use in \Cref{sectionmacphersonconjfull} to shed light on the structure of $R$ in \Cref{dugaldsconjectureshort}. It uses \Cref{envelopedims} and is a generalisation of Proposition 5.2.2 in \cite{cherhrush}. Proposition 5.2.2 in \cite{cherhrush} is essentially about the formula $x=x$, since it concerns the sizes of envelopes, rather than the sizes of definable subsets of envelopes, and arbitrary formulas with parameters arise only as part of the proof. In contrast, \Cref{5.2.2} concerns arbitrary formulas with parameters from the outset and so more complexity arises. We also go into considerably more detail on certain points than in the proof given in \cite{cherhrush}.

\begin{propn}[cf.\ Propsition 5.2.2 in \cite{cherhrush}]\label{5.2.2}
Let $\E$ be an ordered family of envelopes of a Lie coordinatised $\La$-structure $\M$ such that $\dom(\mu)=\dom(\mu')$ for all $(E,\mu),(E',\mu')\in\E$ and such that the parity and signature of orthogonal spaces are constant on the family, where by `ordered family' we mean that for all $(E,\mu),(E',\mu')\in\E$ either $E\subseteq E'$ or $E'\subseteq E$. Let $\mybar{a}\in M^m$ (where $m$ is arbitrary), let $D_{\mybar{a}}\subseteq M$ be an $\La_{\mybar{a}}$-definable set and let $s$ be the size of the common domian of the dimension functions. Then there exists a polynomial $\rho\in\Q[\mathbf{X}_1,\ldots,\mathbf{X}_s]$ and an integer $Q\in\Na$ such that $|D_{\mybar{a}}\cap E|=\rho(\mybar{d^*}(E))$ for all $(E,\mu)\in\E$ with $|E|>Q$ and $\mybar{a}\in E^m$.
\end{propn}

\begin{rem}\label{sigparassump}
We offer a brief explanation of the parity/signature assumption; full details can be found in \secref{21} of \cite{asch}. The parity of a finite-dimensional orthogonal space $V$ refers to $\dim(V)$, distinguishing between odd and even dimension. The signature refers to the quadratic form on $V$, there being only two possibilities (up to equivalence); in the even-dimensional case this is determined by the Witt index and in the odd-dimensional case by the hyperbolic hyperplane. The assumption is important, but its use is restricted to the calculations at the end of the proof, and there only in the orthogonal case.
\end{rem}

\begin{proof}[Proof of \Cref{5.2.2}] Let $\vphi(x,\mybar{a})$ be the $\La_{\mybar{a}}$-formula that defines $D_{\mybar{a}}$. So $D_{\mybar{a}}=\vphi(\M,\mybar{a})$. By \Cref{liecat} and the Ryll-Nardzewski Theorem we may assume without loss of generality that $\vphi(x,\mybar{a})$ defines the set of realisations of a 1-type $r(x)$ over $\mybar{a}$ in $\M$. So $D_{\mybar{a}}=r(\M)$. Also note that since $\mathcal{E}$ is ordered by $\subseteq$, either $|D_{\mybar{a}}\cap E|=\0$ for all $(E,\mu)\in\E$ or there exists $Q\in\Na$ such that $|D_{\mybar{a}}\cap E|\neq\0$ for all $(E,\mu)\in\E$ with $|E|>Q$. In the former case we can set $Q:=0$ and $\rho:=0$. So we henceforce assume that we are in the latter case. With these two assumptions in hand, we are now in a position to start the main line of argument. We proceed by induction on CH-rank.

First suppose that $\rk(D_{\mybar{a}})=0$. Then $D_{\mybar{a}}$ is finite. Let $k:=|D_{\mybar{a}}|$. Since $D_{\mybar{a}}$ is both finite and $\mybar{a}$-definable, $D_{\mybar{a}}\subseteq\acl(\mybar{a})$. Thus, since envelopes are algebraically closed (\Cref{envelopedefn}), $D_{\mybar{a}}\subseteq E$ for all $E\in\E$ with $\mybar{a}\in E^m$. So $|D_{\mybar{a}}\cap E|=|D_{\mybar{a}}|=k$ for all $E\in\E$ with $\mybar{a}\in E^m$. Hence the constant polynomial $\rho:=k$ suffices.

Now consider the case $\rk(D_{\mybar{a}})>0$. Then $D_{\mybar{a}}$ is infinite. Assume as the induction hypothesis that the result holds for any parameter-definable subset of $M$ with CH-rank strictly less than $\rk(D_{\mybar{a}})$. Let $d\in D_{\mybar{a}}$. For a contradiction, suppose that every step in the tree below $d$ is algebraic; that is, if $c_0<c_1<\cdots<c_t=d$ is the chain leading to $d$, where $c_0$ is the root of the tree, then each $c_{i+1}$ is algebraic over its immediate predecessor $c_i$. We claim that $d\in\acl(\0)$.

\hypertarget{mini-induction}We prove this claim. We proceed by induction on $i$ to show that $c_i\in\acl(\0)$ for every $i$, and so in particular $d=c_t\in\acl(\0)$. Since the root is $\0$-definable (\Cref{Lie-isation-defn}), $c_0\in\dcl(\0)\subseteq\acl(\0)$. Now suppose that $c_i\in\acl(\0)$. Then $c_{i+1}\in\acl(\acl(\0))$, since $c_{i+1}\in\acl(c_i)$ by our supposition. But $\acl(\acl(\0))=\acl(\0)$, since algebraic closure is idempotent, and hence $c_{i+1}\in\acl(\0)$. So the claim is proved.

We now use the claim to derive a contradiction. Since $d\in\acl(\0)$, there exists some $\La$-formula $\chi(x)$ such that $\M\models\chi(d)$ and $\chi(\M)$ is finite. So $\chi(x)\in\tp(d/\0)\subseteq\tp(d/\mybar{a})=r(x)$ and hence $D_{\mybar{a}}=r(\M)\subseteq\chi(\M)$ is finite, a contradiction.

So by the contradiction there exists $c\leq d$ such that $c$ is not algebraic over its immediate predecessor. Take $c$ to be minimal, i.e.\ lowest in the tree. By \Cref{Lie-isation-defn} the non-algebraicity of $c$ implies that $c$ lies in a coordinatising geometry $J$, where $J$ is $b$-definable for some $b<c$. The minimality of $c$ implies that $J$ is a projective Lie geometry, since the vector and affine parts of a coordinatising affine Lie geometry lie above the projectivisation of the vector part. Recalling \Cref{ignorequad}, the same argument applies to quadratic geometries: The affine part $Q$ of a coordinatising quadratic geometry, namely the set of quadratic forms on which the vector part $V$ acts by translation, lies above $V$ in the tree, $V$ being a symplectic space. So the minimality of $c$ implies that $J$ is the projectivisation of $V$.

\emph{Case 1:\ The element $b$ is the root.} Then $b\in\dcl(\0)$ and so $J$ is $\0$-definable. We define a set that is central to our argument:
\[
S:=\{(c',d')\in M^2:\tp((c',d')/\mybar{a})=\tp((c,d)/\mybar{a})\}.
\]
Let $S_i$ be the projection of $S$ to the $i^{\mathrm{th}}$ coordinate. Then $S_1$ is the set of realisations of $\tp(c/\mybar{a})$ and $S_2$ is the set of realisations of $\tp(d/\mybar{a})$, as proved in the next paragraph. Then $S_1\subseteq J$, since $c\in J$ and $J$ is $\0$-definable, and $S_2=D_{\mybar{a}}$, since $\tp(d/\mybar{a})=r(x)$.

We prove the claim that $S_1$ is the set of realisations of $\tp(c/\mybar{a})$: If $c'\in S_1$, then it is immediate from the definition of $S$ that $c'\models\tp(c/\mybar{a})$. Now suppose that $c'\models\tp(c/\mybar{a})$. By the Ryll-Nardzewski Theorem, $\M$ is saturated and thus there exists $\sigma\in\Aut(\M/\mybar{a})$ such that $\sigma(c)=c'$; we'll use this trick several more times and henceforth won't cite it explicitly. Thus $(c',\sigma(d))=(\sigma(c),\sigma(d))\in S$ and hence $c'\in S_1$. So the claim is proved. The proof of the claim that $S_2$ is the set of realisations of $\tp(d/\mybar{a})$ proceeds symmetrically.

Let's now consider the intersection of $D_{\mybar{a}}$ with an envelope. So take some arbitrary $(E,\mu)\in\E$ with $|E|>Q$ and $\mybar{a}\in E$. Since $D_{\mybar{a}}\cap E\neq\0$, we may assume without loss of generality that $d\in E$, for if $d\notin E$, then we may take some $d'\in D_{\mybar{a}}\cap E$ and repeat the previous arguments for this new element $d'$.

Define
\[
S_E:=\{(c',d')\in S:d'\in E\}.
\]
We will use this set to calculate the size of $D_{\mybar{a}}\cap E$, but we first need to go over some preliminaries. Let $S_{Ei}$ be the projection of $S_E$ to the $i^{\mathrm{th}}$ coordinate. Then $S_{E2}=S_2\cap E=D_{\mybar{a}}\cap E$. We claim that $S_{E1}=S_1\cap E$.

We prove this claim. Let $c'\in S_{E1}$. Then $(c',d')\in S_E$ for some $d'\in E$. Now, $c'\leq d'$ and so $c'\in\dcl(d')$. Thus, since envelopes are algebraically closed (by definition), $c'\in E$. So $c'\in S_1\cap E$ (since $S_{E1}\subseteq S_1$), as required. Now let $c'\in S_1\cap E$. Let $d''\in D\cap E$. Since $\tp(d''/\mybar{a})=\tp(d/\mybar{a})$, there exists $\sigma\in\Aut(\M/\mybar{a})$ such that $\sigma(d)=d''$. Let $c'':=\sigma(c)$. Then $(c'',d'')\in S_E$. By the same argument used earlier in this paragraph, $c''\in E$. Now, $\tp(c''/\mybar{a})=\tp(c'/\mybar{a})$ and so there exists $\sigma'\in\Aut(\M/\mybar{a})$ such that $\sigma'(c'')=c'$. Now, since envelopes are homogeneous substructures (Lemma 3.2.4 in \cite{cherhrush} and \Cref{defnsmoothapprox}) and $c',c''\in E$, we may assume that $\sigma(E)=E$. Let $d':=\sigma'(d'')$. Then $d'\in E$, since $d''\in E$. Hence $(c',d')\in S_E$ and so $c'\in S_{E1}$, as required. So the claim is proved.

We introduce some further definitions: For $c'\in S_1$ let $c'/S_2:=\{d':(c',d')\in S\}$ and $c'/S_{E2}:=\{d':(c',d')\in S_E\}$, and for $d'\in S_2$ let $d'/S_1:=\{c':(c',d')\in S\}$ and $d'/S_{E1}:=\{c':(c',d')\in S_E\}$. The sizes of the $c'/S_{E2}$ and the $d'/S_{E1}$ are in fact independent of $c'$ and $d'$, as we now show.

First consider some arbitrary $c'\in S_{E1}$. Let $D_{\mybar{a}c}$ be the set of realisations of $\tp(d/\mybar{a}c)$. Then, by the definition of $S$, $D_{\mybar{a}c}=c/S_2$. Let $d'\in c'/S_{E2}$. Then, since $\tp((c',d')/\mybar{a})=\tp((c,d)/\mybar{a})$, there exists $\sigma\in\Aut(\M/\mybar{a})$ such that $\sigma(c',d')=(c,d)$. We claim that $\sigma\colon c'/S_2\to c/S_2$ is a bijection. Injectivity is immediate. It is well-defined, since if $d''\in c'/S_2$, then $\sigma(c',d'')=(c,\sigma(d''))\in S$ and so $\sigma(d')\in c/S_2$. It is surjective, since if $d''\in c/S_2$, then $\sigma^{-1}(c,d'')=(c',\sigma^{-1}(d''))\in S$ and so $\sigma^{-1}(d'')\in c'/S_2$. So the claim is proved. Now, as mentioned previously, envelopes are homogeneous substructures. So, since $d,d'\in E$, we may assume that $\sigma(E)=E$. Thus
\begin{equation}\label{sizecse2}
\begin{aligned}
|c'/S_{E2}|&=|c'/S_2\cap E|\\
&=|c/S_2\cap E|\\
&=|D_{\mybar{a}c}\cap E|
\end{aligned}
\end{equation}
for all $c'\in S_{E1}$.

\hypertarget{uniquec}Now consider some arbitrary $d'\in S_{E2}$. Since $c\leq d$, $c\in\dcl(d)$. Thus, since $\tp(d'/\mybar{a})=\tp(d/\mybar{a})$, there exists a unique $c'\in M$ such that $(c',d')\in S$. But $d'\in E$ and so $(c',d')\in S_E$. Hence
\begin{equation}\label{sizedse1}
|d'/S_{E1}|=1
\end{equation}
for all $d'\in S_{E2}$.

We are now in a position to calculate the size of $S_E$ and thereby also that of $D_{\mybar{a}}\cap E$. Let's first calculate $|S_E|$ in terms of $|S_{E1}|$:
\begin{equation}\label{sizese1}
\begin{aligned}
|S_E|&=\sum_{c'\in S_{E1}}|c'/S_{E2}| \\
&= |S_{E1}|\cdot |D_{\mybar{a}c}\cap E| \quad \text{(by \eqref{sizecse2})}.
\end{aligned}
\end{equation}
And now in terms of $|S_{E2}|$:
\begin{equation}\label{sizese2}
\begin{aligned}
|S_E|&=\sum_{d'\in S_{E2}}|d'/S_{E1}|\\
&= |S_{E2}| \quad \text{(by \eqref{sizedse1})}.
\end{aligned}
\end{equation}
So, since $S_{E2}=D_{\mybar{a}}\cap E$, \eqref{sizese1} and \eqref{sizese2} yield
\begin{equation}\label{sizede}
|D_{\mybar{a}}\cap E|=|S_{E1}|\cdot |D_{\mybar{a}c}\cap E|.
\end{equation}

\hypertarget{S_1poly}First consider $S_{E1}$. We previously proved that $S_{E1}=S_1\cap E$. We also showed that $S_1$ is the set of realisations of $\tp(c/\mybar{a})$ and that $S_1$ is a subset of $J$. By the Ryll-Nardzewski Theorem, $\tp(c/\mybar{a})$ is isolated and so $S_1$ is $\mybar{a}$-definable. So $S_1$ is an $\mybar{a}$-definable subset of a projective geometry. Thus, \hyperlink{S_1polyproof}{as we will show later} (after Case 2), there exists a polynomial $\rho_1\in\Q[\mathbf{X}_1,\ldots,\mathbf{X}_s]$ such that $\rho_1(\mybar{d^*}(E))=|S_1\cap E|$.

Now consider $D_{\mybar{a}c}$, which is a parameter-definable subset of $M$, again by the Ryll-Nardzewski Theorem. We have $\rk(D_{\mybar{a}c})<\rk(D_{\mybar{a}})$, as proved in the following paragraph, and thus by the induction hypothesis there exists a polynomial $\rho_2\in\Q[\mathbf{X}_1,\ldots,\mathbf{X}_s]$ such that $|D_{\mybar{a}c}\cap E|=\rho_2(\mybar{d^*}(E))$.

We prove the claim that $\rk(D_{\mybar{a}c})<\rk(D_{\mybar{a}})$. Let $n:=\rk(D_{\mybar{a}c})$. We previously showed that $D_{\mybar{a}c}=c/S_2$. We also showed that for every $c'\in S_1$ there exists $\sigma\in\Aut(\M/\mybar{a})$ such that $\sigma(c/S_2)=c'/S_2$, which thus means $\rk(c'/S_2)=n$ for every $c'\in S_1$. Define $f\colon D_{\mybar{a}}\to S_1$ by $f(d'):=c'$, where $c'$ is such that $(c',d')\in S$. As we \hyperlink{uniquec}{showed earlier}, for every $d\in S_2$ there is precisely one $c'$ such that $(c',d')\in S$, so $f$ is well-defined. Then, since $f^{-1}(c')=c'/S_2$, we have $\rk(f^{-1}(c'))=n$ for every $c'\in S_1$. Also note that $\rk(S_1)>0$, since $S_1$ is infinite (because $c$ is not algebraic over its immediate predecessor). Thus, taking $D:=D_1:=D_{\mybar{a}}$, $\pi:=\Id$, $D_2:=S_1$ and $f:= f$ in \Cref{rank}, we see that $\rk(D_{\mybar{a}})\geq n+1>\rk(D_{\mybar{a}c})$. So the claim is proved.

Define $\rho:=\rho_1\cdot\rho_2$. Then \eqref{sizede} gives us the desired result:
\[
\begin{aligned}
|D_{\mybar{a}}\cap E|&=|S_{E1}|\cdot |D_{\mybar{a}c}\cap E|\\
&=\rho_1(\mybar{d^*}(E))\cdot\rho_2(\mybar{d^*}(E))\\
&=\rho(\mybar{d^*}(E)).
\end{aligned}
\]
\emph{End of Case 1.}

\emph{Case 2:\ The element $b$ is not the root.} Since $c$ is minimal, $b$ and each element below $b$ (except the root) is algebraic over its immediate predecessor. Thus, by the same \hyperlink{mini-induction}{induction used earlier in the proof}, $b\in\acl(\0)$. Thus, by inspection of \Cref{Lie-isation-defn}, we see that we may add to $\La$ a constant symbol for $b$ without affecting the Lie coordinatising tree. Adding the new constant symbol preserves the inequality $\rk(D_{\mybar{a}c})<\rk(D_{\mybar{a}})$, again since $b\in\acl(\0)$, but it makes $J$ $\0$-definable. We may thus simply repeat the argument given in Case 1 in the extended language $\La_b$. \emph{End of Case 2.}

\hypertarget{S_1polyproof}We now prove our \hyperlink{S_1poly}{earlier claim} that there is a polynomial $\rho_1\in\Q[\mathbf{X}_1,\ldots,\mathbf{X}_s]$ such that $\rho_1(\mybar{d^*}(E))=|S_1\cap E|$. The set $S_1$ is an $\La_{\mybar{a}}$-definable subset of $J$ and thus, since $J$ is fully embedded in $\M$, $S_1$ is $\mybar{a}$-definable in the language of $J$; we may assume that $\mybar{a}$ lies in $J$ by stable embeddedness. We now consider the localisation $J/\mybar{a}$ of $J$ at $\mybar{a}$ (\Cref{localisation}). $J$ fibres over $J/\mybar{a}$, where two elements lie in the same fibre if and only if they have the same algebraic closure over $\mybar{a}$. These fibres all have the same finite size, where this size is determined by $\tp(\mybar{a})$. Now, $S_1$ might not respect these fibres; that is, the intersection of $S_1$ with each fibre might vary in size. However, since the fibres are finite, there are only finitely many possible sizes for these intersections and so we can $\mybar{a}$-definably partition the set of fibres according to size. We then consider the intersection of each part of the partition with $E$: We calculate the size of the base of the fibres, which is a $\0$-definable subset of $J/\mybar{a}$, and then multiply this result by the size of the fibre. We then sum these results to obtain $|S_1\cap E|$. So, in short, by localising $J$ at $\mybar{a}$, it suffices to consider $\0$-definable subsets of projective Lie geometries. It remains to do the explicit calculations in each kind of projective Lie geometry. We use quantifier elimination (\Cref{geoqe}). 

\emph{A projectivisation of a degenerate space.} Projectivisation in this case is trivial. The only $\0$-definable set is the whole space itself. (We can rule out $\0$ because $D_{\mybar{a}}\cap E\neq\0$.) So $S_1=J$. Thus, since $J\cap E=\mu(J)$ (\Cref{envelopedefn}), where $\mu$ is the dimension function of $E$, we have $|S_1\cap E|=d_E(J)$, as required.

\emph{A projectivisation of a pure vector space.} The only $\0$-definable set is again the whole space itself. So $S_1=J$. Thus, going via the approximation of the linear space, which has dimension $\dim\mu(J) +1$, we have
\[
|S_1\cap E|=\frac{q^{\dim\mu(J)+1}-1}{q-1}=q^{\dim\mu(J)}+1=(-\sqrt{q})^{2\dim\mu(J)}+1=d_E(J)^2+1,
\]
as required.

\emph{A projectivisation of a polar space.} This is the same as the vector space case, except that we can define either half of the space or the whole space. If the former, then the answer is the same as that in the vector space case. If the latter, then we multiply this answer by $2$.

\emph{A projectivisation of a symplectic space.} Since there is only one $1$-type, this case is the same as the pure vector space case.

\emph{\hypertarget{unitarycalcs}A projectivisation of a unitary space.} The calculations can be found in the proof of Proposition 5.2.2 in \cite{cherhrush}. Note that it is this case that forces us to consider $(-\sqrt{q})^{\dim\mu(J)}$, rather than just $q^{\dim\mu(J)}$.

\emph{An projectivisation of an orthogonal space.} The calculations can again be found in the proof of Proposition 5.2.2 in \cite{cherhrush}. Note that this is where the assumption regarding constant signature and parity is used (\Cref{sigparassump}). Also note that there is a small typographical error in the calculations: On p.~91 of \cite{cherhrush} it should state $n(2i+j,\alpha)=q^i n(j,\alpha)+q^{j-1}(q^{2i}-q^i)$, the original term $q^i n(i,\alpha)$ being incorrect.

One final note: The calculations for unitary and orthogonal spaces in \cite{cherhrush} are actually done in the linear Lie geometry, rather than in the projectivisation. However, by a similar fibering argument to the one \hyperlink{S_1polyproof}{used earlier} with the localisation, this is sufficient.
\end{proof}


\subsection{Macpherson's conjecture, full version}
\label{sectionmacphersonconjfull}

We introduce the notion of a polynomial exact class, enabling us to state and prove \Cref{dugaldsconjecturefull}, the main result of the present work.

\begin{defn}[Polynomial exact class]\label{polymec}
Let $\La$ be a language and $\C$ a class of finite $\La$-structures. Then $\C$ is a \emph{polynomial exact class} in $\La$ if there exist
\begin{enumerate}[(i)]
\item $R\subseteq\Q[\mathbf{X}_1,\ldots,\mathbf{X}_k]$ for some $k\in\Npos$,
\item $\La$-formulas $\delta_1(\mybar{x}_1,\mybar{y}_1),\ldots,\delta_k(\mybar{x}_k,\mybar{y}_k)$ and
\item $\mybar{a}_1\in M^{l(\mybar{y}_1)},\ldots,\mybar{a}_k\in M^{l(\mybar{y}_k)}$ for each $\M\in\C$
\end{enumerate}
such that $\C$ is an $R$-mec in $\La$ where
\[
h(\M)=h\left(|\delta_1(\M^{l(\mybar{x}_1)},\mybar{a}_1)|,\ldots,|\delta_k(\M^{l(\mybar{x}_1)},\mybar{a}_k)|\right)
\]
for every $h\in R$ and for every $\M\in\C$.
\end{defn}

\begin{rem}
If we replace `$R$-mec' with `$R$-mac' in \Cref{polymec}, then we define a \emph{polynomial asymptotic class}. In this case we allow polynomials with irrational coefficients.

Note that any $1$-dimensional asymptotic class is a polynomial asymptotic class, since we may take $\delta$ to be the $\La$-formula $x=x$ and $h$ to be the polynomial $\mu\mathbf{X}^d$, where $(d,\mu)$ is the dimension--measure pair.
\end{rem}

\begin{exmpl}[Theorem 4.3.2 in \cite{gms}]
The class of finite vector spaces is a polynomial asymptotic class.
\end{exmpl}

\begin{thm}[Macpherson's conjecture, full version]\label{dugaldsconjecturefull}
For any countable\linebreak language $\La$ and for any $d\in\Na^+$ the class $\C(\La,d)$ of all finite $\La$-structures with at most $d$ 4-types is a polynomial exact class in $\La$.
\end{thm}

\begin{proof}
By \Cref{dugaldsconjectureshort} we know that $\C:=\C(\La,d)$ is a multidimensional exact class. It remains to show that the measuring functions are polynomial in the sense of \Cref{polymec}.

Recall our use of \Cref{thmsix} in the proof of \Cref{dugaldsconjectureshort}: We partitioned $\C$ into subclasses $\mathcal{F}_1,\ldots,\mathcal{F}_k$ such that the $\La$-structures in each $\mathcal{F}_i$ smoothly approximate an $\La$-structure $\mathcal{F}_i^*$. By the work in \cite{cherhrush} each $\mathcal{F}_i$ is a class of envelopes for $\mathcal{F}_i^*$, which is Lie coordinatisable. So \Cref{5.2.2} implies that $\C$ is a polynomial exact class, since each coordinatising Lie geometry is fully embedded in and thus (by definition) also definable in $\mathcal{F}_i^*$.

We address some details arising from this proof. Firstly, by the \nameref{projlemmec} (\Cref{projlemmec}) it suffices to consider $\La$-formulas in one object variable, as we do in \Cref{5.2.2}. Secondly, by \Cref{eventuallyiff} the intersection $\vphi(\mathcal{F}_i^*,\mybar{a})\cap \M$ is equal to the relativisation $\vphi(\M,\mybar{a})$ for all $\M\in\mathcal{F}_i$ above some minimum size, so by \Cref{finiteexceptions} it suffices to consider the intersection. Thirdly, since $\C$ is an exact class, rather than just an asymptotic class, the measuring functions are determined by the formula and thus it is not necessary to show that the polynomials given by \Cref{5.2.2} are uniform in the parameter $\mybar{a}$; this point is important because the measuring functions cannot depend on the parameters. Lastly, the hypothesis of constant parity and signature in the statement of \Cref{5.2.2} can be satisfied by partitioning each $\mathcal{F}_i$ into (up to) four subclasses, each with constant parity and signature.
\end{proof}

\begin{rem}~
\begin{thmlist}
\item \Cref{dugaldsconjecturefull} generalises Theorem 3.8 in \cite{macstein1} and Proposition 4.1 in \cite{elwes}.
\item \Cref{dugaldsconjecturefull} allows us to improve \Cref{qeexample}: In the proof of \Cref{qeexample}, set $m=4$. Then the proof shows that each structure in $\C$ has at most $d$ 4-types. So $\C$ is a subclass of $\C(\La,d)$ and thus by \Cref{dugaldsconjecturefull} is a polynomial exact class in $\La$.
\end{thmlist}
\end{rem}


\section{Open questions}
\label{section:questions}
\counterwithin{defn}{section}
\counterwithin{fact}{section}
\counterwithin{question}{section}

We pose a number of questions arising from the present work. In doing so we refer to the important model-theoretic notions of stability and (super)simplicity, which we have so far only mentioned in passing. We do not define these notions, but instead direct the reader to the vast literature on them, \cite{casanovas}, \cite{kim} and \cite{tentziegler} being good introductions. We also consider the notion of homogeneity, which is easier to define:

\begin{defn} An $\La$-structure $\M$ is \emph{homogeneous} if $\M$ is countable and every isomorphism between substructures of $\M$ extends to an automorphism of $\M$.
\end{defn}

Note that the word `homogeneous' is overused in mathematics, especially in model theory. What we call `homogeneous' might be called `ultrahomogeneous' by other authors. See the comment after Definition 2.1.1 in \cite{machomogeneous}.

\begin{fact}\label{homogeneosufact}
Let $\La$ be a finite relational language.
\begin{thmlist}
\item If $\M$ is a homogeneous $\La$-structure, then $\M$ is $\aleph_0$-categorical.
\item If $\M$ is an $\aleph_0$-categorical $\La$-structure, then $\Th(\M)$ has quantifier elimination if and only if $\M$ is homogeneous.
\item If $\M$ is a stable homogeneous $\La$-structure, then $\M$ is $\aleph_0$-stable.
\end{thmlist}
\end{fact}

\begin{question}\label{bipartitequestion}
By \Cref{homogeneosufact} and Corollary 7.4 in \cite{chl}, if $\La$ is a finite relational language and $\M$ is a stable homogeneous $\La$-structure, then $\M$ is smoothly approximable and thus by \Cref{sapproxmecshort} is elementarily equivalent to an ultraproduct of a multidimensional exact class. Does the converse hold? That is, if $\La$ is a finite relational language and $\M$ is a homogeneous $\La$-structure that is elementarily equivalent to an ultraproduct of a multidimensional exact class, then is $\M$ necessarily stable?

Recalling \Cref{paleyrem}, answering this question might shed some light on the role in Theorem 7.5.6 in \cite{cherhrush} of the generic bipartite graph, which is neither stable nor smoothly approximable.
\end{question}

The following two questions were suggested to the present author by Ivan Toma\v{s}i\'{c}:

\begin{question}
What is the relationship between the work of Kraj\'{i}\v{c}ek, Scanlon and others on Euler characteristics and $R$-macs and $R$-mecs? \cite{ks}, \cite{scanlon}, \cite{kraj}, \cite{rytentom}, \cite{ivan}

The notion of a generalised measurable structure, as developed in \cite{amsw}, also appears to be related, but a thorough investigation has yet to be carried out.
\end{question}

\begin{question}
What are the interactions between polynomial exact classes and varieties with a polynomial number of points over finite fields?

The work of Brion and Peyre in \cite{bp} would be a good starting point for research into this question, as it suggests that algebraic varieties homogeneous under a linear algebraic group may provide a generic example of a polynomial exact class.
\end{question}

\subsection*{Acknowledgements}
The concepts explored in this paper were introduced to me by Dugald Macpherson, my patient and supportive PhD supervisor, who provided me with invaluable guidance and scrutiny throughout its composition. I had many pleasant and productive discussions with my friend and collaborator Sylvy Anscombe, while Gregory Cherlin, Immanuel Halupczok, Ehud Hrushovski and Ivan Toma\v{s}i\'{c} provided further useful feedback and insight. Several people on the \TeX{} Stack Exchange helped me overcome some \LaTeX{} issues that arose whilst writing this paper, while my wise and wonderful wife Mandy helped me overcome some non-\LaTeX{} issues that arose whilst writing this paper. I am indebted to the anonymous referee for their careful and thorough report; the present work is much improved for their suggestions.

\bibliography{wolf_references}
\bibliographystyle{ref_style_wolf}

\end{document}